\documentclass[10pt,reqno]{amsart}

\usepackage{amssymb,latexsym,mathrsfs,amsmath}%,tikz,tikz-cd}
\usepackage{amsthm}
\usepackage{graphicx}
\usepackage{tikz}
\usetikzlibrary{arrows}
\usepackage{caption}
\usepackage{subcaption}
\usepackage{etex}

\newcommand{\Kh}{H_{\mathrm{Kh}}}
\newcommand{\BN}{H_{\mathrm{BN}}}

\newcommand{\Z}{\mathbb{Z}}

\newcommand{\cC}{\mathscr{C}}

\newcommand{\X}{\mathcal{X}}

\newcommand{\Q}{\mathbb{Q}}

\newcommand{\K}{\mathbf{k}}
\newcommand{\F}{\mathbb{F}}
\newcommand{\tildcob}{\widetilde{\mathcal{C}ob}^3_\bullet}
\newcommand{\barcob}[1]{\overline{\mathcal{C}ob}^3_{#1}}

\newcommand{\textlbrackdbl}{[\![}
\newcommand{\textrbrackdbl}{]\!]}

\DeclareMathOperator{\Sq}{Sq}

\DeclareMathOperator{\id}{id}
\DeclareMathOperator{\im}{im}

\newtheorem{lemma}{Lemma}[section]

\newtheorem{proposition}[lemma]{Proposition}

\theoremstyle{definition}
\newtheorem{remark}[lemma]{Remark}
\newtheorem{definition}[lemma]{Definition}

\newcommand{\sphere}[3]{
\shade[ball color = gray!40, opacity = 0.3] ({#1},{#2}) circle ({#3});
\draw (#1,#2) circle ({#3});
\draw (#1-#3,#2) arc (180:360:#3 and 0.3*#3);
\draw[dashed] (#1+#3,#2) arc (0:180:#3 and 0.3*#3);
}

\newcommand{\spheredot}[3]{
\sphere{#1}{#2}{#3}
\node at (#1-0.3*#3,#2+0.45*#3) [scale = 1.5*#3] {$\bullet$};
}

\newcommand{\spheredots}[3]{
\sphere{#1}{#2}{#3}
\node at (#1-0.3*#3,#2+0.45*#3) [scale = 1.5*#3] {$\bullet$};
\node at (#1+0.3*#3,#2+0.45*#3) [scale = 1.5*#3] {$\bullet$};
}

\newcommand{\plane}[2]{
\shade[color = gray!40, opacity = 0.3] (#1,#2-0.5) -- (#1+1,#2) -- (#1+1,#2+1.5) -- (#1,#2+1) -- (#1,#2-0.5);
\draw (#1,#2-0.5) -- (#1+1,#2) -- (#1+1,#2+1.5) -- (#1,#2+1) -- (#1,#2-0.5);
}

\newcommand{\planehole}[2]{
\shade[color = gray!40, opacity = 0.3] (#1,#2-0.5) -- (#1+1,#2) -- (#1+1,#2+0.2) arc (270:90:0.3 and 0.3) -- (#1+1,#2+1.5) -- (#1,#2+1) -- (#1,#2-0.5);
\draw (#1,#2-0.5) -- (#1+1,#2) -- (#1+1,#2+0.2) arc (270:90:0.3 and 0.3) -- (#1+1,#2+1.5) -- (#1,#2+1) -- (#1,#2-0.5);
}

\newcommand{\cylinder}[4]{
\shade[ball color = gray!40, opacity = 0.3] (#1,#2) arc (180:360:0.4*#3 and 0.2*#3) -- (#1+0.8*#3,#2+#4*#3) 
arc (0:180:0.4*#3 and -0.2*#3);
\shade[ball color = gray!40, opacity = 0.15] (#1,#2+#4*#3) arc (180:360:0.4*#3 and -0.2*#3) arc (0:180:0.4*#3 and -0.2*#3);
\draw (#1,#2) arc (180:360:0.4*#3 and 0.2*#3) -- (#1+0.8*#3,#2+#4*#3) 
arc (0:180:0.4*#3 and 0.2*#3) -- (#1,#2);
\draw[dashed] (#1+0.8*#3,#2) arc (0:180:0.4*#3 and 0.2*#3);
\draw (#1,#2+#4*#3) arc (180:360:0.4*#3 and 0.2*#3);
}

\newcommand{\bowl}[3]{
\shade[ball color = gray!40, opacity = 0.3] (#1+0.8*#3,#2) arc (0:180:0.4*#3 and -0.2*#3) -- (#1,#2-0.1*#3) arc (180:360:0.4*#3 and 0.4*#3) -- (#1+0.8*#3,#2);
\shade[ball color = gray!40, opacity = 0.15] (#1,#2) arc (180:360:0.4*#3 and -0.2*#3) arc (0:180:0.4*#3 and -0.2*#3);
\draw (#1+0.8*#3,#2) arc (0:180:0.4*#3 and -0.2*#3) -- (#1,#2-0.1*#3) arc (180:360:0.4*#3 and 0.4*#3) -- (#1+0.8*#3,#2);
\draw (#1,#2) arc (180:360:0.4*#3 and -0.2*#3);
}

\newcommand{\bowldot}[3]{
\bowl{#1}{#2}{#3}
\node at (#1+0.3*#3,#2-0.3*#3) [scale = 0.7*#3] {$\bullet$};
}

\newcommand{\bowlud}[3]{
\shade[ball color = gray!40, opacity = 0.3] (#1,#2) arc (180:360:0.4*#3 and 0.2*#3) -- (#1+0.8*#3,#2+0.1*#3) arc (0:180:0.4*#3 and 0.4*#3) -- (#1,#2);
\draw (#1,#2) arc (180:360:0.4*#3 and 0.2*#3) -- (#1+0.8*#3,#2+0.1*#3) arc (0:180:0.4*#3 and 0.4*#3) -- (#1,#2);
\draw[dashed] (#1+0.8*#3,#2) arc (0:180:0.4*#3 and 0.2*#3);
}

\newcommand{\bowluddot}[3]{
\bowlud{#1}{#2}{#3}
\node at (#1+0.3*#3,#2+0.3*#3) [scale = 0.7*#3] {$\bullet$};
}

\newcommand{\pop}[3]{
\shade[ball color = gray!40, opacity = 0.3] (#1,#2) arc (180:360:0.4*#3 and 0.2*#3) to [out = 90, in = 270] (#1+1.6*#3,#2+1*#3) arc (0:180:0.4*#3 and -0.2*#3) arc (0:180:0.4*#3 and -0.4*#3) arc (0:180:0.4*#3 and -0.2*#3) to [out = 270, in = 90] (#1,#2);
\shade[ball color = gray!40, opacity = 0.15] (#1-0.8*#3,#2+#3) arc (180:360:0.4*#3 and -0.2*#3) arc (0:180:0.4*#3 and -0.2*#3);
\shade[ball color = gray!40, opacity = 0.15] (#1+0.8*#3,#2+#3) arc (180:360:0.4*#3 and -0.2*#3) arc (0:180:0.4*#3 and -0.2*#3);
\draw (#1,#2) arc (180:360:0.4*#3 and 0.2*#3) to [out = 90, in = 270] (#1+1.6*#3,#2+1*#3) arc (0:180:0.4*#3 and -0.2*#3) arc (0:180:0.4*#3 and -0.4*#3) arc (0:180:0.4*#3 and -0.2*#3) to [out = 270, in = 90] (#1,#2);
\draw[dashed] (#1,#2) arc (180:360:0.4*#3 and -0.2*#3);
\draw (#1,#2+#3) arc (0:180:0.4*#3 and 0.2*#3);
\draw (#1+1.6*#3,#2+#3) arc (180:360:-0.4*#3 and -0.2*#3);
}

\newcommand{\popud}[3]{
\shade[ball color = gray!40, opacity = 0.3] (#1,#2) arc (180:360:0.4*#3 and 0.2*#3) to [out = 270, in = 90] (#1+1.6*#3,#2-1*#3) arc (0:180:0.4*#3 and -0.2*#3) arc (0:180:0.4*#3 and 0.4*#3) arc (0:180:0.4*#3 and -0.2*#3) to [out = 90, in = 270] (#1,#2);
\shade[ball color = gray!40, opacity = 0.15] (#1,#2) arc (180:360:0.4*#3 and -0.2*#3) arc (0:180:0.4*#3 and -0.2*#3);
\draw (#1,#2) arc (180:360:0.4*#3 and 0.2*#3) to [out = 270, in = 90] (#1+1.6*#3,#2-1*#3) arc (0:180:0.4*#3 and -0.2*#3) arc (0:180:0.4*#3 and 0.4*#3) arc (0:180:0.4*#3 and -0.2*#3) to [out = 90, in = 270] (#1,#2);
\draw (#1,#2) arc (180:360:0.4*#3 and -0.2*#3);
\draw[dashed] (#1,#2-#3) arc (0:180:0.4*#3 and 0.2*#3);
\draw[dashed] (#1+1.6*#3,#2-#3) arc (180:360:-0.4*#3 and -0.2*#3);
}

\newcommand{\poscrossing}[3]{
\draw (#1,#2) -- (#1+#3,#2+#3);
\draw (#1+#3,#2) -- (#1+0.6*#3,#2+0.4*#3);
\draw (#1,#2+#3) -- (#1+0.4*#3,#2+0.6*#3);
}

\newcommand{\negcrossing}[3]{
\draw (#1,#2) -- (#1+0.4*#3,#2+0.4*#3);
\draw (#1+0.6*#3,#2+0.6*#3) -- (#1+#3,#2+#3);
\draw (#1+#3,#2) -- (#1, #2+#3);
}

\newcommand{\smoothingud}[3]{
\draw (#1,#2) to [out = 45, in = 315] (#1,#2+#3);
\draw (#1+#3,#2) to [out = 135, in = 225] (#1+#3,#2+#3);
}

\newcommand{\smoothinglr}[3]{
\draw(#1,#2) to [out = 45, in = 135] (#1+#3,#2);
\draw(#1,#2+#3) to [out = 315, in = 225] (#1+#3,#2+#3);
}

\newcommand{\circlenumber}[1]{
\hspace{-6pt}
\begin{tikzpicture}[baseline=-3pt]%{([yshift=-.75ex]current bounding box.center)}]
\draw (0,0) circle (0.15);
\node at (0,0) [scale = 0.6] {$#1$};
\end{tikzpicture}
\hspace{-6pt}
}

\newcommand{\squarenumber}[5]{
\draw (#1-#3,#2) -- (#1+0.6+#3,#2) -- (#1+0.6+#3,#2+0.6)-- (#1-#3,#2+0.6)-- (#1-#3,#2);
\node at (#1+0.3,#2+0.3) [scale = #5]{$#4$};
\draw (#1+0.1,#2) -- (#1+0.1,#2-0.1);
\draw (#1+0.5,#2) -- (#1+0.5,#2-0.1);
\draw (#1+0.1,#2+0.6) -- (#1+0.1,#2+0.7);
\draw (#1+0.5,#2+0.6) -- (#1+0.5,#2+0.7);
}

\begin{document}
\parindent0em
\setlength\parskip{.1cm}
\thispagestyle{empty}
\title{A fast Algorithm for calculating $S$-Invariants}
\author[Dirk Sch\"utz]{Dirk Sch\"utz}
\address{Department of Mathematical Sciences\\ Durham University\\UK.}
\email{dirk.schuetz@durham.ac.uk}

\begin {abstract}
We use the divide-and-conquer and scanning algorithms for calculating Khovanov cohomology directly on the Lee- or Bar-Natan deformations of the Khovanov complex to give an alternative way to compute Rasmussen $s$-invariants of knots. By disregarding generators away from homological degree $0$ we can considerably improve the efficiency of the algorithm. With a slight modification we can also apply it to a refinement of Lipshitz-Sarkar.
\end {abstract}

\maketitle

\section{Introduction}
The celebrated Rasmussen $s$-invariant \cite{MR2729272} assigns an even integer $s(K)$ to any knot $K$ such that $|s(K)|\leq 2\cdot g_\ast(K)$, where $g_\ast(K)$ is the smooth slice genus of $K$. Its definition is based on the Lee complex \cite{MR2173845}, a deformation of the Khovanov complex \cite{MR1740682} whose differential does not preserve the quantum grading. However, it admits a filtration which is preserved by the differential.
The $s$-invariant is then derived from the filtration.

Rasmussen worked originally over $\Q$, but $s$-invariants have since been generalized to arbitrary fields \cite{MR2320159}, and recently to cohomology operations \cite{MR3189434}. Our algorithm has been designed to deal with these. Another class of generalizations have been considered in \cite{MR3692911}, it would be worthwhile to see if our techniques can be used in that setting.

The computation of $s$-invariants is time consuming, and is usually done by calculating the Khovanov cohomology first, compare \cite{MR2729272, MR2657647}, but is potentially rewarding as it may detect counterexamples to the smooth Poincar\'e conjecture%\footnote{It was claimed in \cite{MR3100886} that this is not the case, but the statement has recently been retracted \cite{2011arXiv1110.1297K}.}
, see \cite{MR2657647}. Whether this is a promising line of attack remains to be seen though.

In \cite{MR2320156} Bar-Natan developed a powerful algorithm to calculate Khovanov cohomology which also made the calculation of $s$-invariants possible for large classes of knots. Lee cohomology \cite{MR2173845} on the other hand has a rather simple form, with all the information for a knot contained in homological degree $0$. This was used in \cite{MR2776617}, where bounds for the $s$-invariant are given based on focussing directly on the Lee complex in homological degree $0$.

Our approach to calculate the $s$-invariant is to use the techniques of \cite{MR2320156} directly on the Lee complex (or rather on a slight variation of this complex) while also focussing on homological degree $0$. With only small adjustments we can use this to calculate the $\Sq^1$-refinement of \cite{MR3189434}. The algorithm we present in this paper has been turned into a computer programme, see \cite{SKnotJob}, and results of calculations are listed in Section \ref{sec:examples}.

Our calculations for the $\Sq^1$-refinement suggest that it detects whether the original Rasmussen invariant differs from the $s$-invariant over $\F_2$. In theory, this need not be the case, but it would be interesting to have a knot where this is actually not the case. For such a knot the $\Sq^1$-refinement would give a better lower bound on the slice genus than the $s$-invariants.

\subsection*{Acknowledgements}

The author would like to thank Andrew Lobb for many useful discussions.

\section{Lee--Bar-Natan cohomology}

Our description of $s$-invariants largely follows \cite[\S 2]{MR3189434}, as this will help with the transition to the $\Sq^1$-refinement.

Let $\F$ be a field and $V$ the vector space over $\F$ generated by two elements $x_+$ and $x_-$. The \em Bar-Natan Frobenius algebra \em is then given by using the multiplication $m\colon V \otimes V \to V$
\[
 m(x_+\otimes x_+) = x_+ \hspace{0.5cm}m(x_+\otimes x_-)=x_- \hspace{0.5cm}m(x_-\otimes x_+)=x_-\hspace{0.5cm} m(x_-\otimes x_-) = x_-,
\]
the comultiplication $\Delta\colon V \to V\otimes V$
\[
 \Delta(x_+) = x_+\otimes x_- + x_-+\otimes x_+- x_+\otimes x_+ \hspace{1cm}\Delta(x_-) = x_-\otimes x_-,
\]
with unit $\imath\colon \F \to V$ and counit $\eta\colon V \to \F$ given by
\[
 \imath(1)=x_+, \hspace{1cm}\eta(x_+)=0\hspace{0.8cm}\eta(x_-)=1.
\]
A link diagram $L$ gives rise to a bigraded cochain complex $C^{\ast,\ast}(L;\F)$ in the usual way, and we call this complex the Bar-Natan complex, see \cite[\S 9.3]{MR2174270}, and also \cite{MR3189434,MR2320159,MR2286127}. A slightly different complex was used by Lee in \cite{MR2173845} which does not work as well over fields of characteristic $2$.

The cohomology $\BN^\ast(L;\F)$ of this cochain complex has been calculated in \cite[Prop.2.3]{MR2320159}, and in particular for a knot $K$ there are exactly two generators, both in homological degree $0$.

The quantum degrees of the generators give rise to a filtration $\mathcal{F}_qC = \mathcal{F}_qC^{\ast,\ast}(L;\F)$ of subcomplexes with
\[
\cdots \subset \mathcal{F}_{q+2}C \subset \mathcal{F}_qC \subset \mathcal{F}_{q-2}C \subset \cdots \subset C^{\ast,\ast}(L;\F).
\]
The \em Rasmussen $s$-invariant \em of a knot $K$ is then defined as
\begin{align*}
s^\F(K) &= \max\{q\in 2\Z+1\,|\, i^\ast \colon H^\ast(\mathcal{F}_qC) \to \BN^\ast (K;\F) \cong \F^2 \mbox{ surjective}\}+1\\
&= \max\{q\in 2\Z-1\,|\, i^\ast \colon H^\ast(\mathcal{F}_qC) \to \BN^\ast (K;\F) \cong \F^2 \mbox{non-zero}\}-1.
\end{align*}
It is shown in \cite[Prop.2.6]{MR3189434} that the two lines give indeed the same number.

Rasmussen \cite{MR2729272} originally worked over $\Q$ and used the Lee complex \cite{MR2173845}, but it is shown in \cite{MR2320159} that we get the same result from the Bar-Natan complex. Seed, compare \cite{MR3189434}, provided examples of knots $K$ for which $s^{\F_2}(K) \not= s^{\Q}(K)$, where $\F_2$ is the field with two elements. The first example is $14n19265$.

The usual technique to calculate $s^\F(K)$ is to work out the Khovanov cohomology of $K$ and read it off the Poincar\'e polynomial as in \cite[Thm.5.1]{MR2657647}, hoping one does not run into ambiguities when applying this theorem. Working out the Khovanov cohomology can be very time-consuming, and currently the fastest known method is Bar-Natan's scanning algorithm \cite{MR2174270}.

We plan to use this algorithm on the Bar-Natan complex rather than the Khovanov complex. This does not seem like an improvement, as the Bar-Natan complex has the same generators as the Khovanov complex, and a more complicated coboundary. But since the cohomology of a knot is concentrated in degree $0$, we can disregard a lot of the information that arises during the construction of the complex. We will make this more precise in the next section. 

First we want to get a local version for the Bar-Natan complex, using the method of \cite{MR2174270}. Recall that a \em pre-additive category \em $\mathfrak{C}$ is a category where the morphism sets abelian groups such that composition is bilinear in the obvious sense.  If it contains a zero-object and a biproduct of any two objects, the category is called an \em additive category\em.

Any pre-additive category can be turned into an additive category by formally adding direct sums of objects (the biproducts), and treating composition and addition of morphisms as for matrices. Details are given in \cite[Def.3.2]{MR2174270}. We denote this additive closure of the pre-additive category $\mathfrak{C}$ as $\mathfrak{M(C)}$.
Furthermore, we can also consider the category $\mathfrak{K(C)}$ of cochain complexes over an additive category $\mathfrak{C}$, again, see \cite{MR2174270} for details.

In \cite{MR2174270} Bar-Natan defines a pre-additive category $\mathcal{C}ob^3(B)$, where $B\subset S^1$ is a finite set of points and the objects are compact $1$-dimensional manifolds $M\subset D^2$ with $\partial M=B$, and the morphisms are free abelian groups generated by cobordisms within the cylinder, up to boundary fixing isotopies. 
A slight variation is given in \cite[\S 11.2]{MR2174270}, denoted $\mathcal{C}ob^3_\bullet(B)$, which has the same objects, but cobordisms are allowed to have markings in the form of finitely many dots, which are allowed to move freely.

%We can treat these categories as graded categories by freely adding a $\Z$-action on the objects, and letting morphism sets ignore the grading. So if $S$ is an object, we denote by $S\{q\}$ the shifted object for $q\in \Z$.

Define a quotient category $\widetilde{\mathcal{C}ob}^3_\bullet(B)$ of $\mathcal{C}ob^3_\bullet(B)$ by adding relations to the morphism groups using the local relations
\begin{equation}\label{eq:dots}
\begin{tikzpicture}[baseline={([yshift=-.5ex]current bounding box.center)}]
\sphere{0}{0}{0.5}
\node at (1,0) {$=0$,};
\spheredot{3}{0}{0.5}
\node at (4,0) {$=1$,};
\plane{5}{-0.5}
\node at (5.5,0.2) [scale = 0.75] {$\bullet$};
\node at (5.5,-0.2) [scale = 0.75] {$\bullet$};
\node at (6.3,0) {$=$};
%\plane{8}{-0.5}
\planehole{6.6}{-0.5}
\node at (7.1,0) [scale = 0.75] {$\bullet$};
\spheredots{7.6}{0}{0.3}
\node at (8,-0.2) {,};
\end{tikzpicture}
\end{equation}
and
\begin{equation}\label{eq:cylrel}
\begin{tikzpicture}[baseline={([yshift=-.5ex]current bounding box.center)}]
\cylinder{0}{0}{1}{2}
\node at (1.5,1) {$+$};
\bowl{2.2}{2}{1}
\bowlud{2.2}{0}{1}
\spheredots{2.6}{1}{0.4}
\node at (3.7,1) {$=$};
\bowldot{4.4}{2}{1}
\bowlud{4.4}{0}{1}
\node at (5.9,1) {$+$};
\bowl{6.6}{2}{1}
\bowluddot{6.6}{0}{1}
\node at (8,0.8) {.};
\end{tikzpicture}
\end{equation}
Notice that setting
\[
\begin{tikzpicture}
\spheredots{0}{0}{0.3}
\node at (0.75,0) {$=0$};
\end{tikzpicture}
\]
recovers $\mathcal{C}ob^3_{\bullet/l}(B)$ from \cite[\S 11.2]{MR2174270}. Relation (\ref{eq:dots}) and (\ref{eq:cylrel}) ensure that any morphism from the emptyset to itself is a linear combination of products of spheres with two dots. In other words, we have $\mathrm{Mor}(\emptyset,\emptyset) \cong \Z[H]$, a polynomial ring in a variable $H$.

Similar relations were introduced in \cite{MR2253455} to give a local model for the Lee complex.

We can turn the categories into graded categories by considering objects together with an integer $q$ which denotes the grading. We thus write $S\{q\}$ for an object in the graded category, where $S$ is still a compact $1$-dimensional manifold embedded in the disc. Unlike in \cite[\S 6]{MR2174270} we will be more restrictive with the morphisms.
We require the basic morphisms birth and death to lower the grading by $1$, while surgeries increase the grading by $1$, and dotting raises the grading by $2$.

The Delooping Lemma of \cite{MR2320156} carries over, but with a different isomorphism.

\begin{lemma}\label{lem:deloop}
If an object $S$ of $\tildcob(B)$ contains a circle $C$, then $S$ is isomorphic to $S'\{1\}\oplus S'\{-1\}$ in $\mathfrak{M}(\tildcob(B))$, where $S'$ is the object obtained by removing the circle $C$ from $S$.
\end{lemma}

\begin{proof}
This is completely analogous to \cite[Pf.\ of Lm.3.1]{MR2320156}, although the first isomorphism $\varphi\colon S \to S'\{1\}\oplus S'\{-1\}$ is slightly different because of relation (\ref{eq:cylrel}). In pictures it can be described as
\[
\begin{tikzpicture}
\draw (0,0) circle (0.3);
\draw[->] (0.5,0) -- (2.5,0);
\node at (3.2,0.5) {$\emptyset \{1\}$};
\node at (3.2,0) {$\oplus$};
\node at (3.2,-0.5) {$\emptyset \{-1\}$};
\draw[->] (4,0) -- (6,0);
\draw (6.5,0) circle (0.3);
\bowldot{1}{1.5}{0.4}
\bowl{1.7}{1.5}{0.4}
\spheredots{1.86}{1.1}{0.16}
\node at (1.5,1.3) {$-$};
\bowl{1}{0.6}{0.4}
\draw (0.8,0.2) arc (270:90:0.15 and 0.75);
\draw (2.2,0.2) arc (270:90:-0.15 and 0.75);
\bowlud{4.5}{0.4}{0.4}
\bowluddot{5.2}{0.4}{0.4}
\node at (4.3,0.45) [scale = 1.2] {$($};
\node at (5.7,0.45) [scale = 1.2] {$)$};
\end{tikzpicture}
\]
where the morphism on the left arrow is a $2\times 1$ matrix, while the morphism on the right arrow is a $1\times 2$ matrix. Notice that we consider composition of cobordisms as going from top to bottom.
\end{proof}

As in \cite{MR2174270} any tangle diagram $T$ gives rise to an object $\textlbrackdbl T\textrbrackdbl$ in $\mathfrak{K}(\tildcob(\partial T))$. Furthermore, for any commutative ring $\K$ we have a functor $\mathcal{F}\colon \mathcal{C}ob^3_\bullet(\emptyset) \to \mathfrak{Mod}_{\K[H]}$ to the category of $\K[H]$-modules, given by $\mathcal{F}(S) = V^{\otimes k}$, where $k$ is the number of components of $S$ and $V=\K[x,H]/(x^2=xH)$, and the tensor product is over $\K[H]$.
As a $\K[H]$-module, $V$ is freely generated by $1$ and $x$. On morphisms we need to declare $\mathcal{F}$ on
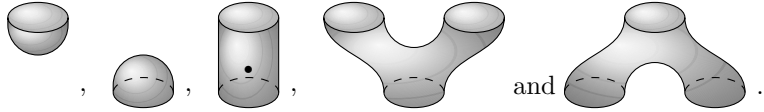
\[
\begin{tikzpicture}
\bowl{0}{1}{1}
\node at (1,0) {,};
\bowlud{1.4}{0}{1}
\node at (2.4,0) {,};
\cylinder{2.8}{0}{1}{1}
\node at (3.2,0.3) [scale = 0.7] {$\bullet$};
\node at (3.8,0) {,};
\pop{5}{0}{1}
\popud{8.2}{1}{1}
\node at (7,0.1) {and};
\node at (10,0) {.};
\end{tikzpicture}
\]
On the generators $1$ and $x$ of $V$ we set
\[
\begin{tikzpicture}
\node at (0,0) {$\mathcal{F}($};
\bowl{0.3}{0.1}{0.4}
\node at (2.8,0) {$)\colon V \to \K[H],\, 1 \mapsto 0,\, x\mapsto 1$};
\node at (0,-0.5) {$\mathcal{F}($};
\bowlud{0.3}{-0.6}{0.4}
\node at (2.2,-0.5) {$)\colon \K[H] \to V ,\, 1 \mapsto 1$};
\node at (0,-1) {$\mathcal{F}($};
\cylinder{0.3}{-1.125}{0.25}{1}
\node at (0.42,-1.03) {$\cdot$};
\node at (2.55,-1) {$)\colon V \to V,\, 1 \mapsto x,\, x\mapsto x^2$};
\node at (0,-1.5) {$\mathcal{F}($};
\pop{0.45}{-1.625}{0.25}
\node at (5.6,-1.5) {$)\colon V\otimes V \to V,\, 1\otimes 1 \mapsto 1,\, 1 \otimes x\mapsto x,\, x\otimes 1 \mapsto x, \, x\otimes x \mapsto x^2$};
\node at (0,-2) {$\mathcal{F}($};
\popud{0.45}{-1.875}{0.25}
\node at (5.26,-2) {$)\colon V \to V \otimes V,\, 1 \mapsto 1\otimes x + x\otimes 1- H \cdot 1\otimes 1, \, x\mapsto x\otimes x$};
\end{tikzpicture}
\]
It is straightforward to check that $\mathcal{F}$ is a well defined functor from $\mathcal{C}ob^3_\bullet(\emptyset) $ to $\mathfrak{Mod}_{\K[H]}$, and furthermore
\[
\begin{tikzpicture}
\node at (0,0) {$\mathcal{F}($};
\spheredots{0.4}{0}{0.15}
\node at (2.45,0) {$)\colon \K[H] \to \K[H], \, 1\mapsto H.$};
\end{tikzpicture}
\]
With this it is easy to see that $\mathcal{F}$ descends to a functor $\mathcal{F}(\emptyset)\colon \tildcob\to \mathfrak{Mod}_{\K[H]}$. Finally, evaluating $H$ to $u\in \K$ gives rise to a functor $\mathcal{E}_u\colon \mathfrak{Mod}_{\K[H]} \to \mathfrak{Mod}_{\K}$. Setting $u=0$ recovers the ordinary Khovanov complex, while setting $u=1$ leads to the Bar-Natan complex described above, using the identifications $x=x_-$ and $1=x_+$.

Adding the relation
\begin{equation}\label{eq:twodotsisone}
\begin{tikzpicture}[baseline={([yshift=-.5ex]current bounding box.center)}]
\spheredots{0}{0}{0.3}
\node at (0.75,0) {$=1$};
\end{tikzpicture}
\end{equation}
to $\tildcob(B)$ allows us to go straight to $\mathfrak{Mod}_\K$ and the Bar-Natan complex, although we lose the grading this way. In practice, we treat the twice-dotted sphere as an identity which raises the grading by $2$. To distinguish it from the identity, we write $I$ for this morphism.

\section{The algorithm}

Bar-Natan's divide-and-conquer algorithm in \cite{MR2320156} can be described as follows. 

Given a complex $C^\ast$ in $\mathfrak{K}(\mathcal{C}ob^3_{\bullet /l}(B))$, the cochain objects are direct summands of objects in $\mathcal{C}ob^3_{\bullet /l}(B)$, which are embedded $1$-dimensional compact manifolds with boundary $B$. Using the Delooping Lemma \cite[Lm.3.1]{MR2320156} we can remove all the circle components in this manifold at the cost of producing more direct summands in the cochain objects. 
The result is a new cochain complex $\bar{C}^\ast$ isomorphic to $C^\ast$, but where the cochain objects have more direct summands.

Now search for direct summands $S_1$ in $\bar{C}^n$ and $S_2$ in $\bar{C}^{n+1}$ such that the matrix entry of the coboundary between these objects is an isomorphism. At this stage one may want to work with a pre-additive category where the morphism sets are vector spaces rather than abelian groups, to increase the chances of finding an isomorphism. 
Then use Gaussian Elimination \cite[Lm.3.2]{MR2320156} to replace $\bar{C}^\ast$ by a chain homotopic cochain complex $\tilde{C}^\ast$ which has fewer direct summands (namely $S_1$ and $S_2$ have been cancelled). Continue to cancel direct summands this way as long as possible to obtain a cochain complex $\hat{C}^\ast$ chain homotopy equivalent to $C^\ast$.

Given a tangle $T$, cut it into two tangles $T_1,T_2$ so that $\textlbrackdbl T \textrbrackdbl = \textlbrackdbl T_1 \textrbrackdbl \otimes \textlbrackdbl T_2 \textrbrackdbl$. Now find $\hat{C}^\ast_1 \simeq \textlbrackdbl T_1 \textrbrackdbl$ and $\hat{C}^\ast_2 \simeq \textlbrackdbl T_2 \textrbrackdbl$ inductively using the previous step and further cutting the tangles. 
Then $\textlbrackdbl T \textrbrackdbl \simeq \hat{C}^\ast_1\otimes \hat{C}^\ast_2$, and we can apply the previous step once more to the right-hand side to get $\textlbrackdbl T \textrbrackdbl \simeq \hat{C}^\ast$. We expect $\hat{C}^\ast$ to have much fewer generators than $\textlbrackdbl T \textrbrackdbl$.

In practice, if both $\hat{C}^\ast_1$ and $\hat{C}^\ast_2$ have a large number of generators, their tensor product can have too many generators to make this efficient. For this reason, \cite[\S 6]{MR2320156} introduced a \em scanning algorithm\em, where the link diagram is described as a list of tangles $T_i$, each of which may only contain one crossing. One then forms $C_k = \hat{C}_{k-1}^\ast \otimes \textlbrackdbl T_k \textrbrackdbl$ and $\hat{C}_k^\ast$ from $C^\ast_k$ via delooping and Gaussian elimination.
This way the second cochain complex in the tensor product always has a small number of generators only, and computer implementations of the scanning algorithm have proven to be rather effective.

Lemma \ref{lem:deloop} ensures we can deloop in $\mathfrak{K}(\tildcob(B))$, and Gaussian elimination \cite[Lm.3.2]{MR2320156} works for general additive categories, so these algorithms work also for the Bar-Natan complex. Of course, since \cite[Prop.2.3]{MR2320159} tells us the cohomology, we are not interested in applying the algorithms directly.

Instead we want to simplify $C^\ast(L;\F)$ to a filtered complex $\hat{C}^\ast$ such that the chain homotopy equivalence provides chain homotopy equivalences between the filtration subcomplexes. To make this precise, we now only consider the scanning algorithm where each tangle is either
\begin{equation}\label{eq:pandm}
\begin{tikzpicture}[baseline={([yshift=-.5ex]current bounding box.center)}]
\node[anchor = east]at (0,0.3) {$P=$};
\poscrossing{0}{0}{0.6}
\node[anchor=west] at (0.8,0.3) {or $M=$};
\negcrossing{2.2}{0}{0.6}
\end{tikzpicture}
\end{equation}

Let us write the resulting bigraded complex $\textlbrackdbl T \textrbrackdbl$ as $C^{\ast,\ast}(T)$ which has two generators
\[
\begin{tikzpicture}
\node[anchor=center] at (0.1,0) {$C^{0,0}(P) = $};
\smoothingud{1}{-0.3}{0.6}
\node[anchor = west] at (1.6,0) {, $C^{1,1}(P) = $};
\smoothinglr{3.7}{-0.3}{0.6}
\node[anchor = west] at (4.3,0) {, and $C^{0,0}(M) = $};
\smoothinglr{7.1}{-0.3}{0.6}
\node[anchor = west] at (7.7,0) {, $C^{1,1}(M) = $};
\smoothingud{9.8}{-0.3}{0.6}
\end{tikzpicture}
\]
and the morphism is the obvious surgery. The distinction between $P$ and $M$ here is somewhat artificial, since we have not chosen an orientation yet, and so $M$ is obtained from $P$ by a rotation. Furthermore, the same rotation identifies the cochain complexes. For now we can ignore this, but later when we consider other tangles built from $P$ and $M$, this will be useful.

Notice that we consider the first degree as the homological degree, while we consider the second degree as the quantum degree. We will slightly change the pre-additive category, and consider $C^{\ast,\ast}(P)$ an object in $\mathfrak{K}(\barcob{\K}(\partial P))$, where $\barcob{\K}(B)$ is the pre-additive category whose objects are the same as in $\mathcal{C}ob^3_\bullet(B)$, but the morphism sets are free $\K$-modules rather than free abelian groups, and we use the relations (\ref{eq:dots}), (\ref{eq:cylrel}) and (\ref{eq:twodotsisone}).
We can also consider this a graded pre-additive category by formally introducing a quantum grading, and assume the generator for $C^{\ast,q}$ having quantum degree $q$.

Now assume our knot diagram $D$ is given as a sequence of tangles $T_1,T_2,\ldots,T_n$ where $n$ is the number of crossings in $D$, and each $T_i$ is either $P$ or $M$. Since $D$ represents a knot, every crossing has a well defined sign $\varepsilon(T_i)\in \{\pm 1\}$. Notice that $\varepsilon(P)=-\varepsilon(M)$ if we assume the orientations to agree on the endpoints. 

Let $n_+$ be the number of crossings with sign $+1$, and $n_-$ be the number of crossings with sign $-1$. Recall that the Khovanov complex of $D$ gets a global shift in homological degree and in quantum degree based on $n_+$ and $n_-$, and it will be convenient for us to apply this shift to the first tangle. We therefore start with
\begin{equation}\label{eq:algstart}
 C_1^{\ast,\ast} = C^{\ast+n_-,\ast+2n_--n_+}(T_1),
\end{equation}
and for $1<i\leq n$ we form
\[
 C_i^{\ast,\ast} = C_{i-1}^{\ast,\ast} \otimes C^{\ast,\ast}(T_i),
\]
using the usual rule for forming tensor products of cochain complexes. Note that this is an object in $\mathfrak{K}(\barcob{\K}(\partial T_i'))$, where $T_i'$ is the obvious sub-diagram of $D$ made up from the tangles $T_1$ to $T_i$. Also, there is a functor $\overline{\mathcal{F}}\colon \barcob{\K}(\emptyset) \to \mathfrak{Mod}_\K$, defined analogously to $\mathcal{F}$, with  $\overline{\mathcal{F}}(C^{\ast,\ast}) = C^{\ast,\ast}(D;\K)$.

Each $C^{\ast,\ast}_i$ has a filtration coming from quantum degrees, given by
\[
 \mathcal{F}_qC_i^{\ast} = \bigoplus_{j\geq q} C^{\ast,j}_i.
\]
and each $\mathcal{F}_qC_i^h$ decomposes as a direct summand of objects in $\barcob{\K}(\partial T_i')$ with quantum grading at least $q$.

The idea is to apply delooping and Gaussian elimination without disturbing the filtrations. There is an immediate problem in that the delooping isomorphism does not respect the filtrations. Gaussian elimination on the other hand behaves well with filtrations, provided the isomorphism along which we cancel respects quantum degrees.

We claim that by carefully performing Gaussian elimination, we get a filtered cochain complex homotopy equivalent to the Bar-Natan complex, so that the filtration subcomplexes are also chain homotopy equivalent. The resulting algorithm can be described as follows. We assume that we have a knot diagram $D$ with $n$ crossings, which are listed as tangles $T_1,\ldots,T_n$.

{\bf Step 1.} Form $C_1^{\ast,\ast}$ using (\ref{eq:algstart}).

{\bf Step 2.} Assuming we have $C_{i-1}^{\ast,\ast}$ for some $i>1$ with $i\leq n$, form
\begin{equation}\label{eq:algtensor}
\tilde{C}_i^{\ast,\ast} = C_{i-1}^{\ast,\ast} \otimes C^{\ast,\ast}(T_i).
\end{equation}
{\bf Step 3.} Form $\bar{C}_i^{\ast,\ast}$ as the result of delooping $\tilde{C}_i^{\ast,\ast}$.

{\bf Step 4.} In $\bar{C}_i^{\ast,\ast}$ search for morphisms $k\cdot \id\colon S_q \to S_q$, where $k\in \K$ is a unit, and $S_q$ is an object in $\barcob{\K}(T_i')$ which is a direct summand in both $\bar{C}_i^{m,q}$ and $\bar{C}_i^{m+1,q}$, and perform Gaussian elimination along such a morphism.

Continue until no such morphisms can be found, and call the resulting cochain complex $C_i^{\ast,\ast}$.

{\bf Step 5.} If $i<n$, repeat from Step 2.

This algorithm ends with a cochain complex $C^{\ast,\ast}_n$ with the property that $\overline{\mathcal{F}}(C^{\ast,\ast}_n)$ is chain homotopy equivalent to the Bar-Natan complex $C^{\ast,\ast}(D;\K)$. But more importantly, the resulting filtration on this complex behaves well with respect to the filtration of the Bar-Natan complex.

\begin{proposition}\label{prop:goodfiltration}
Denote $C^{\ast,\ast} = \overline{\mathcal{F}}(C^{\ast,\ast}_n)$, and let
\[
 \mathcal{F}_qC^\ast = \bigoplus_{j\geq q} C^{\ast,j}.
\]
Then there is a chain homotopy equivalence $\varphi\colon C^{\ast,\ast} \to  C^{\ast,\ast}(D;\K)$ which restricts to chain homotopy equivalences $\varphi_q\colon \mathcal{F}_qC^\ast \to \mathcal{F}_qC^\ast(D;\K)$ for all $q\in \Z$.
\end{proposition}

We prove Proposition \ref{prop:goodfiltration} in Section \ref{sec:proofofthm}. For now recall that we are interested in determining the $s$-invariant of the knot $K$ represented by $D$, so let us specialize to the case $\K=\F$ is a field. Observe that for $q$ odd we get
\[
 \mathcal{F}_q C^n / \mathcal{F}_{q+2} C^n \cong \Kh^{n,q}(K;\F),
\]
the Khovanov cohomology of $K$ with coefficients in $\F$. To see this note that the cohomology of the quotient complex is the Khovanov cohomology by Proposition \ref{prop:goodfiltration}, but the quotient complex has trivial coboundary, for otherwise we could still have done more cancellations in Step 4 for $i=n$.

We now want to calculate the $s$-invariant directly from the definition, but using the filtered complex $C^{\ast,\ast}$. Notice that $C^{\ast,\ast}$ comes with an explicit basis, which comes from the objects of $C^{\ast,\ast}_n$, which after all is a cochain complex over the category $\barcob{\F}(\emptyset)$. We therefore have a basis of each $C^{i,q}$ which we denote by $c^{i,q}_1,\ldots,c^{i,q}_{k_{i,q}}$.

Denote the coboundary of $C^{\ast,\ast}$ by $\delta$. Notice that $\delta(c^{i,q}_j)$ is a linear combination of basis elements $c^{i+1,q'}_l$ with $q'>q$. Since $H^0(C^{\ast,\ast}) \cong \F^2$, we can do Gaussian elimination on basis elements $c^{0,q}_j$ until there are only two left. If $\delta(c^{0,q}_j)\not=0$, we can do Gaussian elimination of $c^{0,q}_j$ against any $c^{1,q'}_l$ appearing in the coboundary with a non-zero coefficient. 
We can also do Gaussian elimination of $c^{0,q}_j$ if it appears in the coboundary of some $c^{-1,q''}_m$ with a non-zero coefficient. Basically, we want to eliminate basis elements of homological degree $0$ until there are only two left, and then read off the $s$-invariant from the quantum degree of the two remaining generators. For this to work, we need to do the cancellations in the right order.

The filtration $\mathcal{F}_qC^\ast \subset C^{\ast,\ast}$ induces maps on cohomology, and \cite[Prop.2.6]{MR3189434} tells us that there is a unique $s\in 2\Z$ such that
\[
 i^\ast\colon H^0(\mathcal{F}_{s-1}C) \to H^0(C)
\]
is surjective, and
\[
 i^\ast\colon H^0(\mathcal{F}_{s+1}C) \to H^0(C)
\]
is non-zero, but not surjective. Furthermore, for $q>s+1$ the map on cohomology is trivial. The correct order to cancel the remaining generators can now be described in the continuation of our algorithm as follows.

{\bf Step 6.} Denote $D_1^\ast = C^{\ast,\ast}$, and choose the basis $c^{i,q}_j$ from above.

{\bf Step 7.} Assume we have $D_{i-1}^\ast$ for some $i>1$ with a basis. If the coboundary $\delta\colon D_{i-1}^0 \to D_{i-1}^1$ is non-zero, choose a basis element $c^{0,q}_j$ with maximal $q$ such that $\delta(c^{0,q}_j)\not=0$, and form $D_i^\ast$ by performing Gaussian elimination of $c^{0,q}_j$ with a basis element $c^{1,q'}_{j'}$. The basis of $D_i^\ast$ is obtained from the basis of $D_{i-1}^\ast$ by removing the two cancelled elements.

{\bf Step 8.} Repeat Step 7 until the coboundary $\delta \colon D^0_i \to D^1_i$ is zero.

{\bf Step 9.} Assume we have $D_{k-1}^\ast$ for some $k>i$ with a basis. If the coboundary $\delta\colon D_{i-1}^{-1} \to D_{i-1}^0$ is non-zero, choose a basis element $c^{0,q}_j$ with minimal $q$ such that $c^{0,q}_j$ appears with non-zero coefficient in $\delta(c^{-1,q'}_{j'})$ for some basis element $c^{-1,q'}_{j'}$. Then form $D_k^\ast$ by performing Gaussian elimination of $c^{0,q}_j$ with $c^{-1,q'}_{j'}$. The basis of $D_k^\ast$ is obtained from the basis of $D_{k-1}^\ast$ by removing the two cancelled elements.

{\bf Step 10.} Repeat Step 9 until the coboundary $\delta \colon D^{-1}_k \to D^0_k$ is zero.

\begin{proposition}\label{prop:readoffsinv}
The final based cochain complex $D_k^\ast$ satisfies $D_k^0=\F^2$, and the two basis elements in homological degree $0$ have quantum degree $s(K)+1$ and $s(K)-1$, respectively.
\end{proposition}

\begin{remark}
Steps 1-10 therefore give an algorithm to calculate the $s$-invariant of a knot over a field $\F$. The first five steps give us a smaller version of the Bar-Natan complex, and the last five steps reduce the $0$-th group until we can read off the invariant.

As it stands, we still calculate the full Khovanov cohomology as a side effect, and since the coboundary of the Bar-Natan complex is more complicated than the Khovanov coboundary, our cancellation technique is less effective. We will see however in Section \ref{sec:analyze} that we can do the same cancellations as if working on the Khovanov complex. The efficiency will still be a bit worse, as we need to keep track of the Bar-Natan deformation.

To regain some efficiency, notice that the final five steps focus on the complex in degree $0$. During the first five steps we operate on objects with positive homological degree which are never used again. It is therefore tempting to simply disregard such elements. The resulting cochain complex will not have the same chain homotopy type, but will have the right cohomology in degree $0$. This can be made to work, and leads to noticable practical improvements.
\end{remark}

The aforementioned disregarding of certain objects can be implemented by slightly changing Step 4. The replacement is given as follows.

{\bf Step 4'.}  In $\bar{C}_i^{\ast,\ast}$ search for morphisms $k\cdot \id\colon S_q \to S_q$, where $k\in \K$ is a unit, and $S_q$ is an object in $\barcob{\K}(T_i')$ which is a direct summand in both $\bar{C}_i^{m,q}$ and $\bar{C}_i^{m+1,q}$ with $m\in \{-2-n+i,\ldots,1\}$, and perform Gaussian elimination along such a morphism.

Continue until no such morphisms can be found, set the $m$-th cochain group to $0$ for $m\notin \{-1-n+i,\ldots,1\}$ and call the resulting cochain complex $C_i^{\ast,\ast}$.

Notice that the final cochain complex $C^{\ast,\ast}_n$ has only non-zero cochain groups in homological degree $-1,0,1$. It also has a filtration coming from the quantum degree so we can look at the maximal $q$ for which $H^0(\mathcal{F}_qC) \to H^0(C)$ is surjective and call this value $s-1$. The next Lemma states that this gives the right value.

\begin{lemma}\label{lem:samesinv}
The value $s$ agrees with the $s$-invariant of the knot.
\end{lemma}

We can therefore calculate the $s$-invariant of a knot with the algorithm obtained by replacing Step 4 with Step 4'.

\section{Proofs}
\label{sec:proofofthm}

It will be useful for us to restate Gaussian elimination from \cite[Lm.3.2]{MR2320156} in a way that makes the chain homotopy equivalences more visible.

\begin{lemma}[Gaussian Elimination]\label{lem:gausselim}
Let $\mathfrak{C}$ be an additive category and $(C^\ast,c^\ast)$ be a cochain complex over $\mathfrak{C}$ such that $C^n=A^n\oplus B^n$, $C^{n+1}=A^{n+1}\oplus B^{n+1}$, and the coboundary $c^n\colon C^n\to C^{n+1}$ is represented by a matrix
\[
 c^n=\begin{pmatrix} \varphi & \delta \\ \gamma & \varepsilon \end{pmatrix}
\]
with $\varphi\colon A^n \to A^{n+1}$ an isomorphism. Then $C^\ast$ is chain homotopy equivalent to a cochain complex $(D^\ast,d^\ast)$ with $D^k=C^k$ for $k\not=n,n+1$, $D^k=B^k$ for $k=n,n+1$, $d^k=c^k$ for $k\not=n-1,n,n+1$. Furthermore, we have a commutative ladder between $C^\ast$ and $D^\ast$ with the vertical morphisms chain homotopy equivalences.
\begin{equation}\label{eq:gaussladder}
\begin{tikzpicture}[baseline={([yshift=-.5ex]current bounding box.center)}]
\node at (0,0) {$C^{n-1}$};
\node at (2.8,0) {$A^n\oplus B^n$};
\node at (6.4,0) {$A^{n+1}\oplus B^{n+1}$};
\node at (9.6,0) {$C^{n+2}$};
\draw[->] (0.5,-0.05) -- (2,-0.05);
\draw[->] (3.6,-0.05) -- (5.3,-0.05);
\draw[->] (7.5,-0.05) -- (9.1,-0.05);
\draw[->] (10.1,-0.05) -- (10.5,-0.05);
\draw[->] (-0.9,-0.05) -- (-0.5,-0.05);
\node at (4.4,0.5) {$\begin{pmatrix}\varphi & \delta \\ \gamma & \varepsilon \end{pmatrix}$};
\node at (1.2,0.5) {$\begin{pmatrix}\alpha \\ \beta \end{pmatrix}$};
\node at (8.3,0.3) {$\begin{pmatrix}\mu & \nu \end{pmatrix}$};

\node at (0,-2) {$C^{n-1}$};
\node at (2.8,-2) {$B^n$};
\node at (6.4,-2) {$B^{n+1}$};
\node at (9.6,-2) {$C^{n+2}$};
\draw[->] (0.5,-2.05) -- (2.4,-2.05);
\draw[->] (3.2,-2.05) -- (5.9,-2.05);
\draw[->] (7,-2.05) -- (9.1,-2.05);
\draw[->] (10.1,-2.05) -- (10.5,-2.05);
\draw[->] (-0.9,-2.05) -- (-0.5,-2.05);
\node at (1.4,-1.8) {$\beta$};
\node at (4.5,-1.8) {$\varepsilon - \gamma \varphi^{-1}\delta$};
\node at (8.1,-1.8) {$\nu$};

\draw[->] (-0.1,-0.2) -- (-0.1,-1.8);
\draw[->] (0.1,-1.8) -- (0.1,-0.2);
\draw[->] (2.7,-0.2) -- (2.7,-1.8);
\draw[->] (2.9,-1.8) -- (2.9,-0.2);
\draw[->] (6.3,-0.2) -- (6.3,-1.8);
\draw[->] (6.5,-1.8) -- (6.5,-0.2);
\draw[->] (9.5,-0.2) -- (9.5,-1.8);
\draw[->] (9.7,-1.8) -- (9.7,-0.2);
\node at (-0.3,-1) {$1$};
\node at (0.3,-1) {$1$};
\node at (2.1,-1) {$\begin{pmatrix} 0 & 1 \end{pmatrix}$};
\node[scale = 0.8] at (3.6,-1) {$\begin{pmatrix} -\varphi^{-1}\delta \\ 1\end{pmatrix}$};
\node[scale = 0.8] at (5.5,-1) {$\begin{pmatrix} -\gamma\varphi^{-1} & 1 \end{pmatrix}$};
\node at (6.9,-1) {$\begin{pmatrix} 0 \\ 1\end{pmatrix}$};
\node at (9.3,-1) {$1$};
\node at (9.9,-1) {$1$};
\end{tikzpicture}
\end{equation}

\end{lemma}

\begin{proof}
Using that $c^{k+1}\circ c^k=0$, it is straightforward to see that the diagram commutes, and the lower row represents a cochain complex. Therefore the downward map $f\colon C^\ast\to D^\ast$ and the upward map $g\colon D^\ast \to C^\ast$ are cochain maps. Clearly $f\circ g=\id_{D^\ast}$, and if we define $H\colon A^{n+1}\oplus B^{n+1} \to A^n \oplus B^n$ by the matrix
\[
 H = \begin{pmatrix} \varphi^{-1} & 0 \\ 0 & 0 \end{pmatrix},
\]
we have $H\circ c^n = \id_{C^n} - g^n\circ f^n$ and $c^n\circ H = \id_{C^{n+1}}-g^{n+1}\circ f^{n+1}$, which means that $f$ and $g$ are chain homotopy equivalences.
\end{proof}

If the cochain complex $C^\ast$ comes with a quantum grading so that for the filtration
\[
 \mathcal{F}_qC^\ast=\bigoplus_{j\geq q} C^{\ast,j}
\]
we have $c^\ast\colon \mathcal{F}_qC^\ast \to \mathcal{F}_qC^{\ast+1}$ for all $q$, that is, each $\mathcal{F}_qC^\ast$ is a subcomplex, we may be able to apply Gaussian elimination to each filtered subcomplex. But this means that we need that $\varphi$ restricts to an isomorphism $\varphi\colon \mathcal{F}_qA^n \to \mathcal{F}_qA^{n+1}$. This is why in Step 4 the morphism is required to preserve $q$.

\begin{proof}[Proof of Proposition \ref{prop:goodfiltration}]
In the first five steps of the algorithm we get a cochain complex via the tensor product, then we deloop it, and then we perform Gaussian elimination, before we repeat these three steps. We claim that we can get the same cochain complex by repeatedly taking the tensor product and delooping, and finally doing (many more) Gaussian eliminations.
But doing the Gaussian eliminations at the last step preserves quantum degrees, so the result follows by the above discussion.

First notice that performing Gaussian elimination from a cochain complex $C^\ast$ to $D^\ast$, and then taking the tensor product with another cochain complex $E^\ast$ has the same result as first taking the tensor product, and then performing Gaussian elimination. To see this, we only need to add $\otimes E^\ast$ and $\otimes 1_{E^\ast}$ in several places of the commutative ladder in Lemma \ref{lem:gausselim}.
Notice that we require $\varphi$ to be a multiple of the identity, and this does not change if we first do the tensor product. The quantum degree of each individual cancellation also remains the same.

We now want to switch the order of performing a Gaussian elimination and the delooping of circles. Assume we perform a Gaussian elimination as in (\ref{eq:gaussladder}). If the delooping does not take place in $A^n,A^{n+1},B^n$ or $B^{n+1}$, it is clear that we can change the order.

If the delooping takes place in the object $A^n$, we can also perform the same delooping in $A^{n+1}$. This is because we assume that the isomorphism is a unit multiple of the identity. The delooped cochain complex in the degrees where we will cancel looks like
\[
\bar{A}_{q+1}^n \oplus \bar{A}_{q-1}^n\oplus B^n\,
\begin{tikzpicture}
\draw[->] (-0.5,0) -- (3.5,0);
\node at (1.5,0.8) {$\begin{pmatrix} k & 0 & ({\color{white}0}\,-{\color{white}0}\,) \delta \\ 0 & k & {\color{white}0}\,\,\delta \\ \gamma\,\, {\color{white}0} & \gamma\,\, {\color{white}0} & \varepsilon \end{pmatrix}$};
\bowldot{1.85}{1.3}{0.3}
\bowl{2.45}{1.3}{0.3}
\bowl{2.1}{0.9}{0.3}
\bowlud{0.25}{0.275}{0.3}
\bowluddot{1.125}{0.275}{0.3}
\end{tikzpicture}
\,\bar{A}_{q+1}^{n+1} \oplus \bar{A}_{q-1}^n \oplus B^n
\]
with $k\in\K$ a unit. Gauss elimination results in
\[
 B^n\,
\begin{tikzpicture}
 \draw[->] (0,0) -- (2,0);
\node at (1,0.2) {$\varepsilon - k^{-1} \gamma\delta$};
\end{tikzpicture}
\,B^{n+1}
\]
which is the same result as if we had first cancelled, with no further delooping.

If the delooping takes place in $B^n$, delooping results in
\[
A^n \oplus \bar{B}_{q+1}^n\oplus \bar{B}_{q-1}^n\,
\begin{tikzpicture}
\draw[->] (0,0) -- (3,0);
\node at (1.5,0.525) {$\begin{pmatrix} k & \delta\,\,{\color{white}0} & \delta\,\,{\color{white}0} \\ \gamma & \varepsilon\,\,{\color{white}0} & \varepsilon\,\,{\color{white}0}\end{pmatrix}$};
\bowlud{1.4}{0.275}{0.3}
\bowluddot{2.2}{0.275}{0.3}
\bowlud{1.4}{0.68}{0.3}
\bowluddot{2.2}{0.68}{0.3}
\end{tikzpicture}
\,A^{n+1} \oplus B^{n+1}
\]
and after Gauss elimination we have
\[
\bar{B}_{q+1}^n\oplus \bar{B}_{q-1}^n\,
\begin{tikzpicture}
\draw[->] (0,0) -- (4,0);
\node at (2,0.25) {$(\varepsilon - k^{-1} \gamma\delta)\circ \begin{pmatrix} {\color{white}00} & {\color{white}00} \end{pmatrix}$};
\bowlud{2.55}{0.175}{0.3}
\bowluddot{3.35}{0.175}{0.3}
\end{tikzpicture}
\,B^{n+1}
\]
which is the same as if we first cancelled and then delooped. The final case where we deloop in $B^{n+1}$ is similar.
\end{proof}

\begin{proof}[Proof of Proposition \ref{prop:readoffsinv}]
Notice that each complex $D_j^\ast$ has a filtration, and we can form its own $s$-invariant $s(D_j)$. Then $s(D_1)=s^\F(K)$, and we want to show that $s(D_{j+1})=s(D_j)$ for all $j=1,\ldots,k-1$.

Consider Step 7, and the basis element $c^{0,q}_j$. Since we cancel it with an element $c^{1,q'}_{j'}$, the chain homotopy equivalence $D_{i-1} \to D_i$ induces chain homotopy equivalences $\mathcal{F}_rD_{i-1} \to \mathcal{F}_rD_i$ for all $r \leq q$. For $r>q$ we still have $H^0(\mathcal{F}_rD_i)=H^0(\mathcal{F}_rD_{i-1})$, since no basis element $c^{0,r}_l$ has non-zero coboundary by the choice of $q$.
It follows that Step 7 does not change the $s$-invariant.

In Step 9 we have $H^0(\mathcal{F}_qD_{k-1})$ is a quotient of $D_{k-1}^0$ obtained by disregarding basis elements $c^{0,q'}$ with $q'<q$, and more relations coming from $\delta(c^{-1,r})$ with $r\geq q$. If $c^{0,q}$ is being cancelled with $c^{-1,q'}$ in Step 9, we have that
\[
 \delta (k\cdot c^{-1,q'}) = c^{0,q} -x,
\]
for some $k\in \F$ and $x\in D_{k-1}^0$. By the choice of $q$ being minimal we get $x\in \mathcal{F}_qD^0_{k-1}$, so for $r\leq q$ the image of $H^0(\mathcal{F}_rD_k)$ in $\F^2=H^0(D_k)$ is the same as the image of $H^0(\mathcal{F}_rD_{k-1})$. For $r>q$ the filtrations are unchanged, hence so are the images. So again the $s$-invariant does not change during Step 9.
\end{proof}

\begin{proof}[Proof of Lemma \ref{lem:samesinv}]
Let $C^\ast_i$ be a cochain complex over $\barcob{\F}(\partial T_i')$, and let $D_i^\ast$ be obtained from $C^\ast_i$ by setting $D_i^k = C_i^k$ for $k\in \{-2-n+i,\ldots,1\}$ and $0$ otherwise, and $\delta_{D_i}=\delta_{C_i}$ in the range where this makes sense.

Now form $\tilde{C}^\ast_{i+1}=C_i^\ast \otimes C^\ast(T_{i+1})$ and $\tilde{D}^\ast_{i+1}=D_i^\ast \otimes C^\ast(T_{i+1})$. Since $C^\ast(T_{i+1})$ is concentrated in homological degrees $0$ and $1$, we have $\tilde{C}^k_{i+1}=\tilde{D}^k_{i+1}$ for $k\in \{-1-n+i,\ldots,1\}$, and $\tilde{D}^k_{i+1}$ is a direct summand of $\tilde{C}^k_{i+1}$ for $k\in \{-2-n+i,2\}$. The coboundaries mostly agree, and differ only slightly towards the ends of the complex, compare Figure \ref{fig:complexdiff}.
\begin{figure}[ht]
\begin{tikzpicture}
\node[color=lightgray] at (0,0) {$C^{k-1}_i\otimes B$};
\node[color=lightgray] at (0,0.6) {$\oplus$};
\node at (0,1.2) {$C^k_i\otimes A$};
\node at (2.3,0) {$C^k_i\otimes B$};
\node at (2.3,0.6) {$\oplus$};
\node at (2.3,1.2) {$C^{k+1}_i\otimes A$};
\node at (5.5,0) {$C^{-1}_i\otimes B$};
\node at (5.5,0.6) {$\oplus$};
\node at (5.5,1.2) {$C^0_i\otimes A$};
\node at (7.8,0) {$C^0_i\otimes B$};
\node at (7.8,0.6) {$\oplus$};
\node at (7.8,1.2) {$C^1_i\otimes A$};
\node at (10.1,0) {$C^1_i\otimes B$};
\node[color=lightgray] at (10.2,0.6) {$\oplus$};
\node[color=lightgray] at (10.1,1.2) {$C^2_i\otimes A$};
\draw[color=lightgray, ->] (-1.3,0) -- (-0.9,0);
\draw[color=lightgray, ->] (-1.3,1.2) -- (-0.9,1.2);
\draw[color=lightgray, ->] (-1.3,0.5) -- (-0.9,0.2);
\draw[color=lightgray, ->] (0.9,0) -- (1.6,0);
\draw[->] (0.7,1.2) -- (1.4,1.2);
\draw[->] (0.7,1) -- (1.6,0.2);
\draw[->] (3,0) -- (3.4,0);
\draw[->] (3.1,1.2) -- (3.4,1.2);
\draw[->] (3.1,1) -- (3.4,0.8);
\draw[->] (4.3,0) -- (4.7,0);
\draw[->] (4.3,1.2) -- (4.8,1.2);
\draw[->] (4.3,0.5) -- (4.7,0.2);
\draw[->] (6.3,0) -- (7.1,0);
\draw[->] (6.2,1.2) -- (7.1,1.2);
\draw[->] (6.2,1) -- (7.1,0.2);
\draw[->] (8.5,0) -- (9.4,0);
\draw[color=lightgray, ->] (8.5,1.2) -- (9.4,1.2);
\draw[->] (8.5,1) -- (9.4,0.2);
\draw[color=lightgray, ->] (10.8,0) -- (11.2,0);
\draw[color=lightgray, ->] (10.8,1.2) -- (11.2,1.2);
\draw[color=lightgray, ->] (10.8,1) -- (11.2,0.7);
\node at (3.9,0.6) {$\cdots$};
\end{tikzpicture}
\caption{The cochain complex $\tilde{D}^\ast_{i+1}$, and the cochain complex $\tilde{C}^\ast_{i+1}$, whose extra terms are indicated in gray. We write $C^{\ast}(T_{i+1})$ as $A \to B$.}
\label{fig:complexdiff}
\end{figure}
We can deloop both complexes, calling them $\bar{D}^\ast_{i+1}$ and $\bar{C}^\ast_{i+1}$, but we still get that they agree for $k\in \{-1-n+i,\ldots,1\}$ and behave as in Figure \ref{fig:complexdiff} at the ends.

We now perform Gaussian elimination on $\bar{D}^\ast_{i+1}$ to get a cochain complex $\hat{D}^\ast_{i+1}$, and do the exact same eliminations on $\bar{C}^\ast_{i+1}$ to get a cochain complex $C^\ast_{i+1}$. We could possibly do more Gaussian eliminations on $C^\ast_{i+1}$. But we will not do this, as now $\hat{D}^\ast_{i+1}$ and $C^\ast_{i+1}$ agree in the range between $k\in \{-2-n+(i+1),\ldots,1\}$.

Cutting $\hat{D}^\ast_{i+1}$ by setting the object in degrees $-2-n+i$ and $2$ equal to $0$ gives rise to a cochain complex $D^\ast_{i+1}$ which is obtained from $C^\ast_{i+1}$ in the same way that $D^\ast_i$ was obtained from $C^\ast_i$.

If we begin with $C_1^\ast$ from (\ref{eq:algstart}), the final complex $D^\ast_n$ is exactly the result from using the first five steps of the modified algorithm, and $C^\ast_n$ is a cochain complex chain homotopy equivalent to the Bar-Natan complex. Since the filtrations are preserved and the two complexes agree in degrees $-1,0,1$, we get a commutative diagram
\[
\begin{tikzpicture}
\node at (0,0) {$H^0(\mathcal{F}_q D_n)$};
\node at (2.5,0) {$H^0(D_n)$};
\node at (0,-1) {$H^0(\mathcal{F}_q C_n)$};
\node at (2.5,-1) {$H^0(C_n)$};
\draw[->] (0.9,0) -- (1.8,0);
\draw[->] (0.9,-1) -- (1.8,-1);
\draw[->] (0,-0.25) -- (0,-0.75);
\draw[->] (2.5,-0.25) -- (2.5,-0.75);
\end{tikzpicture}
\]
Notice that the inclusion $D^\ast_n\subset C^\ast_n$ is not a cochain map, but behaves like one in degree $0$, leading to the vertical maps being actual identities. It follows that both complexes lead to the same $s$-invariant, with the $s$-invariant coming from $C^\ast_n$ being the $s$-invariant of the knot.
\end{proof}

\section{Analyzing cancellations in the Bar-Natan complex}
\label{sec:analyze}

Our rule for performing Gauss elimination sounds rather more restrictive than the one in \cite{MR2320156}. But given two objects $S,S'$ in $\mathcal{C}ob^3_{\bullet /l}(B)$ where both manifolds do not contain circle components, there are not too many isomorphism between them. Surgeries are not isomorphisms, nor is the dotting of an arc, and there is no point in considering dotted spheres since we already look at morphisms up to relations.
So the typical isomorphism is coming from the product cobordism, that is, the identity.

We now want to check that we can assume a similar amount of cancellation opportunities on the Bar-Natan complex as on the Khovanov complex. %This argument will not be completely rigorous, but is supported by computer calculations we present in Section \ref{sec:examples}.

Given a surgery between two arcs
\[
\begin{tikzpicture}
\smoothingud{0}{0}{0.5}
\smoothinglr{1.6}{0}{0.5}
\node at (2.5,0) [scale = 0.7] {$\{+1\}$};
\draw[->] (0.65,0.25) -- (1.45,0.25);
\end{tikzpicture}
\]
tensoring with an object from the next crossing might lead to the same surgery between two arcs, or it might lead to at least one circle in one of the objects. If we only get one new circle, it may be on the left or on the right. So we can get
\[
\begin{tikzpicture}
\smoothingud{0}{0}{0.5}
\draw (0.5,0.5) to [out = 45, in = 135] (1,0.5);
\draw (1,0.5) to [out = 315, in = 45] (1,0);
\draw (1,0) to [out = 225, in = 315] (0.5,0);
\smoothinglr{2.2}{0}{0.5}
\draw (2.7,0.5) to [out = 45, in = 135] (3.2,0.5);
\draw (3.2,0.5) to [out = 315, in = 45] (3.2,0);
\draw (3.2,0) to [out = 225, in = 315] (2.7,0);
\draw[->] (1.25,0.25) -- (2.05,0.25);
\node at (3.6,0) [scale = 0.7] {$\{+1\}$};
\node at (4.15,0.25) {or};
\smoothingud{4.5}{0}{0.5}
\draw (5,0.5) to [out = 45, in =315] (5,1);
\draw (5,1) to [out = 135, in = 45] (4.5,1);
\draw (4.5,1) to [out = 225, in = 135] (4.5,0.5);
\smoothinglr{6.1}{0}{0.5}
\draw (6.6,0.5) to [out = 45, in =315] (6.6,1);
\draw (6.6,1) to [out = 135, in = 45] (6.1,1);
\draw (6.1,1) to [out = 225, in = 135] (6.1,0.5);
\node at (7,0) [scale = 0.7] {$\{+1\}$};
\draw[->] (5.15,0.25) -- (5.95,0.25);
\end{tikzpicture}
\]
With delooping, this turns into
\[
\begin{tikzpicture}
\draw (0,0) to [out = 45, in = 315] (0,0.5);
\node at (0.4,0) [scale = 0.7] {$\{-1\}$};
\draw (0,1) to [out = 45, in = 315] (0,1.5);
\node at (0.4,1) [scale = 0.7] {$\{+1\}$};
\draw (2,0.5) to [out = 45, in = 315] (2,1);
\node at (2.4,0.5) [scale = 0.7] {$\{+1\}$};
\draw (1.1,0.2) to [out = 45, in = 315] (1.1,0.4);
\node at (1.14,0.3) [scale = 0.5] {$\bullet$};
\draw[->] (0.5,0.25) -- (1.8,0.65);
\draw[->] (0.5,1.25) -- node [above, scale = 0.7] {$\id$} (1.8, 0.85);
\node at (3,0.75) {or};
\draw (3.5,0.5) to [out = 45, in = 135] (4,0.5);
\draw (5.5,1) to [out = 45, in = 135] (6,1);
\draw (5.5,0) to [out = 45, in = 135] (6,0);
\node at (6.4,1) [scale = 0.7] {$\{+2\}$};
\draw (4.375,1.15) to [out = 45, in = 135] (4.575,1.15);
\node at (4.475,1.18) [scale = 0.5] {$\bullet$};
\draw[->] (4.2,0.85) -- node [above, scale = 0.7] {$\,-I$} (5.3,1.25);
\draw[->] (4.2,0.65) -- node [below, scale = 0.7] {$\id$} (5.3,0.25);
\end{tikzpicture}
\]
Recall that $I$ is multiplication by a twice-dotted sphere. In both cases we will be able to perform one Gaussian elimination, and this is exactly what we would get if we were working on the Khovanov complex. Also, notice that
\[
\begin{tikzpicture}
\draw (0,0) to [out = 45, in = 315] (0,0.5);
\draw (1.5,0) to [out = 45, in = 315] (1.5,0.5);
\node at (1.9,0) [scale = 0.7] {$\{+2\}$};
\draw[->] (0.4,0.25) -- (1.3,0.25);
\draw (0.8,0.3) to [out = 45, in = 315] (0.8,0.5);
\node at (0.84,0.4) [scale = 0.5] {$\bullet$};
\node at (2.7,0.25) {and};
\draw (3.2,0) to [out = 45, in = 315] (3.2,0.5);
\draw (4.7,0) to [out = 45, in = 315] (4.7,0.5);
\node at (5.1,0) [scale = 0.7] {$\{+2\}$};
\draw[->] (3.6,0.25) -- node [above, scale = 0.7] {$I$} (4.5,0.25);
\end{tikzpicture}
\]
turn into
\[
\begin{tikzpicture}
\node at (0,0) {$\emptyset \{-1\}$};
\node at (0,1) {$\emptyset \{+1\}$};
\node at (3,0) {$\emptyset \{+1\}$};
\node at (3,1) {$\emptyset \{+3\}$};
\draw[->] (0.6,0) -- node [below, scale = 0.7] {$I$} (2.4,0);
\draw[->] (0.6,1) -- node [above, scale = 0.7] {${\color{white}00}-{\color{white}00}=0$} (2.4,1);
\draw[->] (0.6,0.9) -- node [above, sloped, scale = 0.7] {$\id$} (2.4,0.1);
\spheredots{1}{1.175}{0.1}
\spheredots{1.55}{1.175}{0.1}
\node at (4,0.5) {and};
\node at (5,0) {$\emptyset \{-1\}$};
\node at (5,1) {$\emptyset \{+1\}$};
\node at (8,0) {$\emptyset \{+1\}$};
\node at (8,1) {$\emptyset \{+3\}$};
\draw[->] (5.6,0) -- node [below, scale = 0.7] {$I$} (7.4,0);
\draw[->] (5.6,1) -- node [above, scale = 0.7] {$I$} (7.4,1);
\draw[->] (5.6,0.9) -- node [above, sloped, scale = 0.7] {$0$} (7.4,0.1);
\end{tikzpicture}
\]
On the left hand side we can perform one Gaussian elimination in the same way as this is possible in the Khovanov complex. The right hand side does not occur in the Khovanov situation, and we cannot perform Gaussian elimination, as the morphism does not preserve the quantum degree.

From this we expect a similar amount of cancellations on the Bar-Natan complex as on the Khovanov complex. Because of the extra morphisms occuring it may be possible that we do not get the same amount of opportunities to cancel. Furthermore, since we need to keep track of more morphisms, we expect this to be less efficient on the whole of the Bar-Natan complex. 
But our calculations suggest that disregarding generators of unnecessary homological degree more than make up for this shortcoming. 

If we stack crossings as in (\ref{eq:pandm}) on top of each other, we can do cancellations directly. It is not clear that this improves efficiency, but it can be combined with \cite{2017arXiv171001857L} to make the $\Sq^2$-refinement below more calculation friendly. Let us define the tangles
\[
\begin{tikzpicture}
\squarenumber{0}{0}{0}{1}{0.9}
\node at (0.9,0.3) {$=$};
\poscrossing{1.2}{0.05}{0.5}
\squarenumber{2.4}{0}{0}{-1}{0.9}
\node at (3.3,0.3) {$=$};
\negcrossing{3.6}{0.05}{0.5}
\squarenumber{4.9}{0}{0.2}{n+1}{0.9}
\node at (6,0.3) {$=$};
\squarenumber{6.3}{-0.2}{0}{n}{0.9}
\poscrossing{6.4}{0.5}{0.4}
\squarenumber{7.8}{0}{0.25}{-n-1}{0.9}
\node at (8.95,0.3) {$=$};
\squarenumber{9.25}{0.2}{0}{-n}{0.9}
\negcrossing{9.35}{-0.3}{0.4}
\end{tikzpicture}
\]
for $n\geq 1$. 

\begin{proposition}\label{prop:glueddiags}
Let $n\geq 1$. We have 
\begin{enumerate}
\item $C^{\ast,\ast}(\,
\begin{tikzpicture}[baseline = 2pt, scale = 0.5]
\squarenumber{0}{0}{0}{2n}{0.6}
\end{tikzpicture}
\,)$ is chain homotopy equivalent to
\[
\begin{tikzpicture}
\smoothingud{-0.6}{0}{0.6}
\smoothinglr{1}{0}{0.6}
\node at (1.85,0) [scale = 0.5] {$\{+1\}$};
\smoothinglr{2.04}{0.375}{0.15}
\smoothinglr{2.4}{0.375}{0.15}
\node at (2.115,0.485) [scale = 0.4] {$\bullet$};
\node at (2.475,0.4) [scale = 0.4] {$\bullet$};
\smoothinglr{3}{0}{0.6}
\node at (3.85,0) [scale = 0.5] {$\{+3\}$};
\smoothinglr{3.9}{0.365}{0.15}
\smoothinglr{4.25}{0.365}{0.15}
\node at (3.975,0.475) [scale = 0.4] {$\bullet$};
\node at (4.325,0.385) [scale = 0.4] {$\bullet$};
\smoothinglr{5}{0}{0.6}
\node at (5.85,0) [scale = 0.5] {$\{+5\}$};
\smoothinglr{6.04}{0.375}{0.15}
\smoothinglr{6.4}{0.375}{0.15}
\node at (6.115,0.485) [scale = 0.4] {$\bullet$};
\node at (6.475,0.4) [scale = 0.4] {$\bullet$};
\node at (7.3,0.3) {$\cdots$};
\smoothinglr{8.04}{0.375}{0.15}
\smoothinglr{8.4}{0.375}{0.15}
\node at (8.115,0.485) [scale = 0.4] {$\bullet$};
\node at (8.475,0.4) [scale = 0.4] {$\bullet$};
\smoothinglr{9}{0}{0.6}
\node at (10,0) [scale = 0.5] {$\{4n+1\}$};
\draw[->] (0.2,0.3) -- (0.8,0.3);
\draw[->] (1.8,0.3) -- node [above, scale = 0.6] {$-$} (2.8,0.3);
\draw[->] (3.8,0.3) -- node [above, scale = 0.5] {$\,\,\,\,\,\, + \,\,\,\,\,\, - I$} (4.8,0.3);
\draw[->] (5.8,0.3) -- node [above, scale = 0.6] {$-$} (6.8,0.3);
\draw[->] (7.8,0.3) -- node [above, scale = 0.6] {$-$} (8.8,0.3);
\end{tikzpicture}
\]
\item $C^{\ast,\ast}(\,
\begin{tikzpicture}[baseline = 2pt, scale = 0.5]
\squarenumber{0}{0}{0.25}{2n+1}{0.5}
\end{tikzpicture}
\,)$ is chain homotopy equivalent to
\[
\begin{tikzpicture}
\smoothingud{-0.6}{0}{0.6}
\smoothinglr{1}{0}{0.6}
\node at (1.85,0) [scale = 0.5] {$\{+1\}$};
\smoothinglr{2.04}{0.375}{0.15}
\smoothinglr{2.4}{0.375}{0.15}
\node at (2.115,0.485) [scale = 0.4] {$\bullet$};
\node at (2.475,0.4) [scale = 0.4] {$\bullet$};
\smoothinglr{3}{0}{0.6}
\node at (3.85,0) [scale = 0.5] {$\{+3\}$};
\smoothinglr{3.9}{0.365}{0.15}
\smoothinglr{4.25}{0.365}{0.15}
\node at (3.975,0.475) [scale = 0.4] {$\bullet$};
\node at (4.325,0.385) [scale = 0.4] {$\bullet$};
\smoothinglr{5}{0}{0.6}
\node at (5.85,0) [scale = 0.5] {$\{+5\}$};
\smoothinglr{6.04}{0.375}{0.15}
\smoothinglr{6.4}{0.375}{0.15}
\node at (6.115,0.485) [scale = 0.4] {$\bullet$};
\node at (6.475,0.4) [scale = 0.4] {$\bullet$};
\node at (7.3,0.3) {$\cdots$};
\smoothinglr{7.9}{0.365}{0.15}
\smoothinglr{8.25}{0.365}{0.15}
\node at (7.975,0.475) [scale = 0.4] {$\bullet$};
\node at (8.325,0.385) [scale = 0.4] {$\bullet$};
\smoothinglr{9}{0}{0.6}
\node at (10,0) [scale = 0.5] {$\{4n+3\}$};
\draw[->] (0.2,0.3) -- (0.8,0.3);
\draw[->] (1.8,0.3) -- node [above, scale = 0.6] {$-$} (2.8,0.3);
\draw[->] (3.8,0.3) -- node [above, scale = 0.5] {$\,\,\,\,\,\, + \,\,\,\,\,\, - I$} (4.8,0.3);
\draw[->] (5.8,0.3) -- node [above, scale = 0.6] {$-$} (6.8,0.3);
\draw[->] (7.8,0.3) -- node [above, scale = 0.5] {$\,\,\,\,\,\, + \,\,\,\,\,\, - I$} (8.8,0.3);
\end{tikzpicture}
\]
\item $C^{\ast,\ast}(\,
\begin{tikzpicture}[baseline = 2pt, scale = 0.5]
\squarenumber{0}{0}{0.15}{-2n}{0.6}
\end{tikzpicture}
\,)$ is chain homotopy equivalent to
\[
\begin{tikzpicture}
\smoothinglr{0}{0}{0.6}
\node at (1.05,0) [scale = 0.5] {$\{-2n+1\}$};
\smoothinglr{1.04}{0.375}{0.15}
\smoothinglr{1.4}{0.375}{0.15}
\node at (1.115,0.485) [scale = 0.4] {$\bullet$};
\node at (1.475,0.4) [scale = 0.4] {$\bullet$};
\node at (2.3,0.3) {$\cdots$};
\smoothinglr{3.04}{0.375}{0.15}
\smoothinglr{3.4}{0.375}{0.15}
\node at (3.115,0.485) [scale = 0.4] {$\bullet$};
\node at (3.475,0.4) [scale = 0.4] {$\bullet$};
\smoothinglr{4}{0}{0.6}
\node at (5.05,0) [scale = 0.5] {$\{+2n-5\}$};
\smoothinglr{4.9}{0.365}{0.15}
\smoothinglr{5.25}{0.365}{0.15}
\node at (4.975,0.475) [scale = 0.4] {$\bullet$};
\node at (5.325,0.385) [scale = 0.4] {$\bullet$};
\smoothinglr{6}{0}{0.6}
\node at (7.05,0) [scale = 0.5] {$\{+2n-3\}$};
\smoothinglr{7.04}{0.375}{0.15}
\smoothinglr{7.4}{0.375}{0.15}
\node at (7.115,0.485) [scale = 0.4] {$\bullet$};
\node at (7.475,0.4) [scale = 0.4] {$\bullet$};
\smoothinglr{8}{0}{0.6}
\node at (9.05,0) [scale = 0.5] {$\{+2n-1\}$};
\smoothingud{9.6}{0}{0.6}
\draw[->] (0.8,0.3) -- node [above, scale = 0.6] {$-$} (1.8,0.3);
\draw[->] (2.8,0.3) -- node [above, scale = 0.6] {$-$} (3.8,0.3);
\draw[->] (4.8,0.3) -- node [above, scale = 0.5] {$\,\,\,\,\,\, + \,\,\,\,\,\, - I$} (5.8,0.3);
\draw[->] (6.8,0.3) -- node [above, scale = 0.6] {$-$} (7.8,0.3);
\draw[->] (8.8,0.3) -- (9.4,0.3);
\node at (10.5,0) [scale = 0.5] {$\{+2n\}$};
\end{tikzpicture}
\]
\item $C^{\ast,\ast}(\,
\begin{tikzpicture}[baseline = 2pt, scale = 0.5]
\squarenumber{0}{0}{0.35}{-2n-1}{0.5}
\end{tikzpicture}
\,)$ is chain homotopy equivalent to
\[
\begin{tikzpicture}
\smoothinglr{0}{0}{0.6}
\node at (0.9,0) [scale = 0.5] {$\{-2n\}$};
\smoothinglr{0.9}{0.365}{0.15}
\smoothinglr{1.25}{0.365}{0.15}
\node at (0.975,0.475) [scale = 0.4] {$\bullet$};
\node at (1.325,0.385) [scale = 0.4] {$\bullet$};
\node at (2.3,0.3) {$\cdots$};
\smoothinglr{3.04}{0.375}{0.15}
\smoothinglr{3.4}{0.375}{0.15}
\node at (3.115,0.485) [scale = 0.4] {$\bullet$};
\node at (3.475,0.4) [scale = 0.4] {$\bullet$};
\smoothinglr{4}{0}{0.6}
\node at (5.05,0) [scale = 0.5] {$\{+2n-4\}$};
\smoothinglr{4.9}{0.365}{0.15}
\smoothinglr{5.25}{0.365}{0.15}
\node at (4.975,0.475) [scale = 0.4] {$\bullet$};
\node at (5.325,0.385) [scale = 0.4] {$\bullet$};
\smoothinglr{6}{0}{0.6}
\node at (7.05,0) [scale = 0.5] {$\{+2n-2\}$};
\smoothinglr{7.04}{0.375}{0.15}
\smoothinglr{7.4}{0.375}{0.15}
\node at (7.115,0.485) [scale = 0.4] {$\bullet$};
\node at (7.475,0.4) [scale = 0.4] {$\bullet$};
\smoothinglr{8}{0}{0.6}
\node at (8.9,0) [scale = 0.5] {$\{+2n\}$};
\smoothingud{9.6}{0}{0.6}
\draw[->] (0.8,0.3) -- node [above, scale = 0.5] {$\,\,\,\,\,\, + \,\,\,\,\,\, - I$} (1.8,0.3);
\draw[->] (2.8,0.3) -- node [above, scale = 0.6] {$-$} (3.8,0.3);
\draw[->] (4.8,0.3) -- node [above, scale = 0.5] {$\,\,\,\,\,\, + \,\,\,\,\,\, - I$} (5.8,0.3);
\draw[->] (6.8,0.3) -- node [above, scale = 0.6] {$-$} (7.8,0.3);
\draw[->] (8.8,0.3) -- (9.4,0.3);
\node at (10.65,0) [scale = 0.5] {$\{+2n+1\}$};
\end{tikzpicture}
\]
\end{enumerate}
\end{proposition}

\begin{proof}
The proof is done by induction, with the start given by the ordinary surgery complex. We have to do four induction steps, but as they are all very similar, we only show how (2) is derived from (1).

We have $C^{\ast,\ast}(\,
\begin{tikzpicture}[baseline = 2pt, scale = 0.5]
\squarenumber{0}{0}{0.25}{2n+1}{0.5}
\end{tikzpicture}
\,) = C^{\ast,\ast}(\,
\begin{tikzpicture}[baseline = 2pt, scale = 0.5]
\squarenumber{0}{0}{0}{2n}{0.6}
\end{tikzpicture}
\,) \otimes C^{\ast,\ast}(\,
\begin{tikzpicture}[baseline = 2pt, scale = 0.5]
\squarenumber{0}{0}{0}{1}{0.6}
\end{tikzpicture}
\,)$, so by induction this is chain homotopy equivalent to
\[
\begin{tikzpicture}
\smoothingud{-0.6}{0}{0.6}
\smoothinglr{1}{0}{0.6}
\node at (1.85,0) [scale = 0.5] {$\{+1\}$};
\smoothinglr{2.04}{0.375}{0.15}
\smoothinglr{2.4}{0.375}{0.15}
\node at (2.115,0.485) [scale = 0.4] {$\bullet$};
\node at (2.475,0.4) [scale = 0.4] {$\bullet$};
\smoothinglr{5}{0}{0.6}
\node at (6.05,0) [scale = 0.5] {$\{+4n-1\}$};
\smoothinglr{6.04}{0.375}{0.15}
\smoothinglr{6.4}{0.375}{0.15}
\node at (6.115,0.485) [scale = 0.4] {$\bullet$};
\node at (6.475,0.4) [scale = 0.4] {$\bullet$};
\node at (3.3,0.3) {$\cdots$};
\smoothinglr{3.9}{0.365}{0.15}
\smoothinglr{4.25}{0.365}{0.15}
\node at (3.975,0.475) [scale = 0.4] {$\bullet$};
\node at (4.325,0.385) [scale = 0.4] {$\bullet$};
\smoothinglr{7}{0}{0.6}
\node at (8.05,0) [scale = 0.5] {$\{+4n+1\}$};
\smoothinglr{1}{-1.4}{0.6}
\node at (1.85,-1.4) [scale = 0.5] {$\{+1\}$};
\smoothinglr{3}{-1.4}{0.6}
\draw (3.3,-1.1) circle (0.1);
\node at (3.85,-1.4) [scale = 0.5] {$\{+2\}$};
\draw (4.04,-1.045) to [out = 45, in = 135] (4.19,-1.045);
\draw (4.115,-0.95) circle (0.035);
\draw (4.04,-0.855) to [out = 315, in = 225] (4.19,-0.855);
\draw (4.4,-1.045) to [out = 45, in = 135] (4.55,-1.045);
\draw (4.475,-0.95) circle (0.035);
\draw (4.4,-0.855) to [out = 315, in = 225] (4.55,-0.855);
\node at (4.15,-0.95) [scale = 0.4] {$\bullet$};
\node at (4.475,-1.02) [scale = 0.4] {$\bullet$};
\node at (5.3,-1.1) {$\cdots$};
\draw (5.9,-1.065) to [out = 45, in = 135] (6.05,-1.065);
\draw (5.975,-0.97) circle (0.035);
\draw (5.9,-0.875) to [out = 315, in = 225] (6.05,-0.875);
\draw (6.25,-1.065) to [out = 45, in = 135] (6.4,-1.065);
\draw (6.325,-0.97) circle (0.035);
\draw (6.25,-0.875) to [out = 315, in = 225] (6.4,-0.875);
\node at (6.01,-0.97) [scale = 0.4] {$\bullet$};
\node at (6.325,-1.04) [scale = 0.4] {$\bullet$};
\smoothinglr{7}{-1.4}{0.6}
\draw (7.3,-1.1) circle (0.1);
\node at (7.95,-1.4) [scale = 0.5] {$\{+4n\}$};
\draw (8.04,-1.045) to [out = 45, in = 135] (8.19,-1.045);
\draw (8.115,-0.95) circle (0.035);
\draw (8.04,-0.855) to [out = 315, in = 225] (8.19,-0.855);
\draw (8.4,-1.045) to [out = 45, in = 135] (8.55,-1.045);
\draw (8.475,-0.95) circle (0.035);
\draw (8.4,-0.855) to [out = 315, in = 225] (8.55,-0.855);
\node at (8.15,-0.95) [scale = 0.4] {$\bullet$};
\node at (8.475,-1.02) [scale = 0.4] {$\bullet$};
\smoothinglr{9}{-1.4}{0.6}
\draw (9.3,-1.1) circle (0.1);
\node at (10.05,-1.4) [scale = 0.5] {$\{+4n+2\}$};
\draw[->] (0.2,0.3) -- (0.8,0.3);
\draw[->] (1.8,0.3) -- node [above, scale = 0.6] {$-$} (2.8,0.3);
\draw[->] (3.8,0.3) -- node [above, scale = 0.5] {$\,\,\,\,\,\, + \,\,\,\,\,\, - I$} (4.8,0.3);
\draw[->] (5.8,0.3) -- node [above, scale = 0.6] {$-$} (6.8,0.3);
\draw[->] (-0.2,-0.2) -- node [above, scale = 0.6] {$+$} (0.8,-0.9);
\draw[->] (1.8,-0.2) -- node [above, scale = 0.6] {$-$} (2.8,-0.9);
\draw[->] (5.8,-0.2) -- node [above, scale = 0.6] {$-$}(6.8,-0.9);
\draw[->] (7.8,-0.2) -- node [above, scale = 0.6] {$+$}(8.8,-0.9);
\draw[->] (1.8,-1.1) -- (2.8,-1.1);
\draw[->] (3.8,-1.1) -- node [above, scale = 0.6] {$-$} (4.8,-1.1);
\draw[->] (5.8,-1.1) -- node [above, scale = 0.5] {$\,\,\,\,\,\, + \,\,\,\,\,\, - I$} (6.8,-1.1);
\draw[->] (7.8,-1.1) -- node [above, scale = 0.6] {$-$} (8.8,-1.1);
\end{tikzpicture}
\]
After delooping we get
\[
\begin{tikzpicture}
\smoothingud{-0.6}{0}{0.6}
\smoothinglr{1}{0}{0.6}
\node at (1.85,0) [scale = 0.5] {$\{+1\}$};
\smoothinglr{2.04}{0.375}{0.15}
\smoothinglr{2.4}{0.375}{0.15}
\node at (2.115,0.485) [scale = 0.4] {$\bullet$};
\node at (2.475,0.4) [scale = 0.4] {$\bullet$};
\smoothinglr{5}{0}{0.6}
\node at (6.05,0) [scale = 0.5] {$\{+4n-1\}$};
\smoothinglr{6.04}{0.375}{0.15}
\smoothinglr{6.4}{0.375}{0.15}
\node at (6.115,0.485) [scale = 0.4] {$\bullet$};
\node at (6.475,0.4) [scale = 0.4] {$\bullet$};
\node at (3.3,0.3) {$\cdots$};
\smoothinglr{3.9}{0.365}{0.15}
\smoothinglr{4.25}{0.365}{0.15}
\node at (3.975,0.475) [scale = 0.4] {$\bullet$};
\node at (4.325,0.385) [scale = 0.4] {$\bullet$};
\smoothinglr{7}{0}{0.6}
\node at (8.05,0) [scale = 0.5] {$\{+4n+1\}$};
\smoothinglr{1}{-1.4}{0.6}
\node at (1.85,-1.4) [scale = 0.5] {$\{+1\}$};
\smoothinglr{3}{-1.4}{0.6}
\node at (3.85,-1.4) [scale = 0.5] {$\{+3\}$};
\node at (5.3,-1.1) {$\cdots$};
\smoothinglr{7}{-1.4}{0.6}
\node at (8.05,-1.4) [scale = 0.5] {$\{+4n+1\}$};
\smoothinglr{8.15}{-0.4}{0.15}
\smoothinglr{8.35}{-1.025}{0.15}
\node at (8.225,-0.29) [scale = 0.4] {$\bullet$};
\node at (8.425,-1) [scale = 0.4] {$\bullet$};
\smoothinglr{9}{-1.4}{0.6}
\node at (10.05,-1.4) [scale = 0.5] {$\{+4n+3\}$};
\smoothinglr{3}{-2.8}{0.6}
\node at (3.85,-2.8) [scale = 0.5] {$\{+1\}$};
\node at (5.3,-2.8) {$\cdots$};
\smoothinglr{7}{-2.8}{0.6}
\node at (8.05,-2.8) [scale = 0.5] {$\{+4n-1\}$};
\smoothinglr{9}{-2.8}{0.6}
\node at (10.05,-2.8) [scale = 0.5] {$\{+4n+1\}$};
\draw[->] (0.2,0.3) -- (0.8,0.3);
\draw[->] (1.8,0.3) -- node [above, scale = 0.6] {$-$} (2.8,0.3);
\draw[->] (3.8,0.3) -- node [above, scale = 0.5] {$\,\,\,\,\,\, + \,\,\,\,\,\, - I$} (4.8,0.3);
\draw[->] (5.8,0.3) -- node [above, scale = 0.6] {$-$} (6.8,0.3);
\draw[->] (-0.2,-0.2) -- (0.8,-0.9);
\draw[->] (1.8,-0.2) -- (2.8,-0.9);
\draw[->] (5.8,-0.2) -- (6.8,-0.9);
\draw[->] (7.8,-0.2) -- node [above, sloped, scale = 0.5] {$\,\,\,\,\,\, - I$} (8.8,-0.9);
\draw[->] (1.8,-0.3) -- (2.8,-2.2);
\draw[->] (5.8,-0.3) -- (6.8,-2.2);
\draw[->] (7.8,-0.3) -- node [above, sloped, near end, scale = 0.5] {$\id$} (8.8,-2.2);
\draw[->] (1.8,-1.1) -- (2.8,-1.1);
\draw[->] (3.8,-1.1) -- (4.8,-1.1);
\draw[->] (5.8,-1.1) -- (6.8,-1.1);
\draw[->] (7.8,-1.1) -- node [above, scale = 0.6] {$0-\,\,\,\,\,\,$} (8.8,-1.1);
\draw[->] (1.8,-1.6) -- node [above, sloped, scale = 0.5] {$\id$} (2.8,-2.3);
\draw[->] (3.8,-1.6) -- node [above, sloped, scale = 0.5] {$\id-0$} (4.8,-2.3);
\draw[->] (5.8,-1.6) -- node [above, sloped, scale = 0.5] {$\id+0-0$} (6.8,-2.3);
\draw[->] (7.8,-1.6) -- node [above, sloped, scale = 0.5] {$\id-0$} (8.8,-2.3);
\draw[->] (3.8,-2.5) -- (4.8,-2.5);
\draw[->] (5.8,-2.5) -- (6.8,-2.5);
\draw[->] (7.8,-2.5) -- (8.8,-2.5);
\end{tikzpicture}
\]
We can now cancel the elements in the third row with the corresponding element in the middle row, starting from the left. The top row stays as it is, with the last object of the middle row becoming the object of top homological degree. By Gaussian elimination the coboundary is exactly as predicted.
\end{proof}

\section{The Lipshitz-Sarkar refinements of the $s$-invariant}

In \cite{MR3230817} Lipshitz and Sarkar construct a suspension spectrum for a link whose reduced singular cohomology agrees with Khovanov cohomology of this link, which gives rise to stable cohomology operations on Khovanov cohomology. Furthermore, in \cite{MR3189434} they use such cohomology operations to refine the $s$-invariant, leading to new lower bounds for the slice genus of a knot.

To describe these refinements, assume that $\alpha\colon \tilde{H}^\ast(\cdot;\F) \to \tilde{H}^{\ast+n_\alpha}(\cdot; \F)$ is a stable cohomology operation on singular cohomology. If $K$ is a knot, denote by $\X^q_K$ the suspension spectrum defined in \cite{MR3230817} such that $\tilde{H}^i(\X^q_K;\F) = \Kh^{i,q}(K;\F)$.

We thus have a cohomology operation $\alpha\colon \Kh^{i,q}(K;\F)\to \Kh^{i+n_\alpha,q}(K,\F)$ for every $q$. Also, recall that $\Kh^{\ast,q}(K;\F)\cong H^\ast(\mathcal{F}_q/\mathcal{F}_{q+2};\F)$, where we now write simply $\mathcal{F}_q$ for the filtration on the Bar-Natan complex. Therefore we have a zig-zag of maps
\[
 \Kh^{-n_\alpha,q}(K;\F) \stackrel{\alpha}{\longrightarrow} \Kh^{0,q}(K;\F) \longleftarrow H^0(\mathcal{F}_q;\F) \longrightarrow \BN^0(K;\F)\cong \F^2.
\]
Consider the following configurations related to this zig-zag:
\begin{equation}\label{eq:fullconfigs}
\begin{tikzpicture}[baseline={([yshift=-.5ex]current bounding box.center)}]
\node at (0,0) {$\langle \tilde{a},\tilde{b}\rangle$};
\node at (3,0) {$\langle \hat{a}, \hat{b} \rangle$};
\node at (6,0) {$\langle a,b\rangle$};
\node at (9,0) {$\langle \bar{a},\bar{b}\rangle$};
\node at (0,1) {$\Kh^{-n_\alpha,q}(K;\F)$};
\node at (3,1) {$\Kh^{0,q}(K;\F)$};
\node at (6,1) {$H^0(\mathcal{F}_q;\F)$};
\node at (9,1) {$\BN^0(K;\F)$};
\node at (0,2) {$\langle \tilde{a}\rangle$};
\node at (3,2) {$\langle \hat{a}\rangle$};
\node at (6,2) {$\langle a\rangle$};
\node at (9,2) {$\langle \bar{a}\rangle\not=0$};
\draw[right hook->] (0,0.25) -- (0,0.75);
\draw[right hook->] (3,0.25) -- (3,0.75);
\draw[right hook->] (6,0.25) -- (6,0.75);
\draw[-] (9.05,0.25) -- (9.05,0.75);
\draw[-] (8.95,0.25) -- (8.95,0.75);
\draw[right hook->] (0,1.75) -- (0,1.25);
\draw[right hook->] (3,1.75) -- (3,1.25);
\draw[right hook->] (6,1.75) -- (6,1.25);
\draw[right hook->] (9,1.75) -- (9,1.25);
\draw[->] (0.5,0) -- (2.5,0);
\draw[<-] (3.5,0) -- (5.5,0);
\draw[->] (6.5,0) -- (8.5,0);
\draw[->] (0.4,2) -- (2.6,2);
\draw[<-] (3.4,2) -- (5.6,2);
\draw[->] (6.4,2) -- (8.3,2);
\draw[->] (1.1,1) -- node[above] {$\alpha$} (2.1,1);
\draw[<-] (3.95,1) -- node[above] {$p$} (5.15,1);
\draw[->] (6.85,1) -- node[above] {$i^\ast$} (8.1,1);
\end{tikzpicture}
\end{equation}

\begin{definition}
Call an odd integer $q$ \em $\alpha$-half-full\em, if there exist $a\in H^0(\mathcal{F}_q;\F)$ and $\tilde{a}\in Kh^{-n_\alpha,q}(K;\F)$ such that $p(a)=\alpha(\tilde{a})$, and such that $i^\ast(a)=\bar{a}\not=0$. That is, there exists a configuration as in the upper two rows of (\ref{eq:fullconfigs}).

Call an odd integer $q$ \em $\alpha$-full\em, if there exist $a,b\in H^0(\mathcal{F}_q;\F)$ and $\tilde{a},\tilde{b}\in Kh^{-n_\alpha,q}(K;\F)$ such that $p(a)=\alpha(\tilde{a})$, $p(b)=\alpha(\tilde{b})$, and $i^\ast(a),i^\ast(b)$ generate $\BN^0(K;\F)$. That is, there exists a configuration as in the lower two rows of (\ref{eq:fullconfigs}).
\end{definition}

We note that while $i^\ast(a)$ and $i^\ast(b)$ have to be non-zero, it is allowed that $p(a)$ or $p(b)$ are zero.

With the concept of $\alpha$-full and $\alpha$-half-full, Lipshitz and Sarkar \cite{MR3189434} define their refinements of the $s$-invariant as
\begin{definition}
Let $K$ be a knot and $\alpha$ a stable cohomology operation on singular cohomology with coefficients in $\F$, then $r^\alpha_+,r^\alpha_-, s^\alpha_+, s^\alpha_-\in \Z$ are defined as follows.
\begin{align*}
r^\alpha_+(K) &= \max\{q\in 2\Z+1\,|\, q \mbox{ is $\alpha$-half-full}\}+1\\
s^\alpha_+(K) &= \max\{q\in 2\Z+1\,|\, q \mbox{ is $\alpha$-full}\}+3.
\end{align*}
If $\overline{K}$ denotes the mirror of $K$, we also set
\begin{align*}
r^\alpha_-(K) &= -r^\alpha_+(\overline{K})\\
s^\alpha_-(K) &= -s^\alpha_+(\overline{K}).
\end{align*}
We can put them all together and write
\[
 s^\alpha(K) = (r^\alpha_+(K),s^\alpha_+(K),r^\alpha_-(K),s^\alpha_-(K)).
\]
\end{definition}

The motivation for these refinements is that each of these numbers provides a lower bound for the slice genus of $K$, provided these numbers evaluate the unknot to $0$ (which is automatically satisfied for $n_\alpha>0$), and that in some cases they provide better bounds than the $s$-invariant, see \cite[Thm.1,Thm.2]{MR3189434}.

It is shown in \cite[Lm.4.2]{MR3189434} that each of these new invariants differs from $s^\F(K)$ by at most $2$. Indeed, once $i^\ast\colon H^0(\mathcal{F}_q;\F) \to \BN^0(K;\F)$ is the $0$-map, $q$ can no longer be $\alpha$-half-full, and if $i^\ast$ is not surjective, $q$ cannot be $\alpha$-full. Hence
\[
 r^\alpha_+(K),s^\alpha_+(K) \leq s^\F(K)+2.
\]
Also,
\[
 r^\alpha_+(K),s^\alpha_+(K) \geq s^\F(K),
\]
as for small values of $q$ the desired configurations can be achieved using $0$ for $\tilde{a}$ and $\tilde{b}$. So to calculate $r^\alpha_+(K)$ one has to check whether $s^\F(K)+1$ is $\alpha$-half-full, and to calculate $s^\alpha_+(K)$ one has to check whether $s^\F(K)-1$ is $\alpha$-full.

Let us denote
\[
 R(\alpha) = \im \alpha \cap \im p \subset \Kh^{0,s^\F(K)+1}(K;\F),
\]
and
\[
 S(\alpha) = \im \alpha \cap \im p \subset \Kh^{0,s^\F(K)-1}(K;\F).
\]
We want to give a criterion for checking whether any of the refinements differ from the $s$-invariant, which can be easily combined with our algorithm. For notational purposes, let us write $\mathcal{F}_{-\infty}$ for the Bar-Natan complex. There is a commutative diagram
\begin{equation}\label{eq:commdiag}
\begin{tikzpicture}[baseline={([yshift=-.5ex]current bounding box.center)}]
\node[anchor=east] at (0,1.6) {$H^0(\mathcal{F}_q;\F)$};
\node[anchor=west] at (1.6,1.6) {$H^0(\mathcal{F}_{-\infty};\F)=\BN^0(K;\F)$};
\node[anchor=east] at (0,0) {$\Kh^{0,q}(K;\F)=H^0(\mathcal{F}_q/\mathcal{F}_{q+2};\F)$};
\node[anchor=west] at (1.6,0) {$H^0(\mathcal{F}_{-\infty}/\mathcal{F}_{q+2};\F)$};
\draw[->] (0.1,1.6) -- node[above] {$i^\ast$} (1.5,1.6);
\draw[->] (0.1,0.0) -- node[above] {$j^\ast$} (1.5,0.0);
\draw[->] (-0.9,1.35) -- node[right] {$p$} (-0.9,0.3);
\draw[->] (2.5,1.35) -- (2.5,0.3);
\end{tikzpicture}
\end{equation}

\begin{proposition}\label{prop:readoffrefine}
Let $K$ be a knot and $\alpha$ a stable cohomology operation on singular cohomology with coefficients in $\F$.
\begin{enumerate}
\item We have $r^\alpha_+(K)=s^\F(K)+2$ if and only if there is $\hat{a}\in R(\alpha)$ such that $0\not=j^\ast(\hat{a})\in H^0(\mathcal{F}_{-\infty}/\mathcal{F}_{s^\F(K)+3};\F)$.
\item We have $s^\alpha_+(K)=s^\F(K)+2$ if and only if there is $\hat{b}\in S(\alpha)$ such that
$0\not=j^\ast(\hat{b})\in H^0(\mathcal{F}_{-\infty}/\mathcal{F}_{s^\F(K)+1};\F)$.
\end{enumerate}
\end{proposition}

\begin{proof}
Assume that $r^\alpha_+(K)=s^\F(K)+2$ which means that $q=s^\F(K)+1$ is $\alpha$-half-full. By assumption, there exists $\hat{a}=\alpha(\tilde{a})=p(a)$ with $0\not=i^\ast(a)\in \BN^0(K;\F)$. Consider the long exact sequence corresponding to $\mathcal{F}_{q+2}\subset \mathcal{F}_{-\infty}$:
\begin{equation}\label{eq:longexact}
 \cdots \longrightarrow H^0(\mathcal{F}_{q+2};\F) \stackrel{i^\ast}{\longrightarrow} \BN^0(K;\F) \stackrel{p}{\longrightarrow} H^0(\mathcal{F}_{-\infty}/\mathcal{F}_{q+2};\F) \longrightarrow \cdots
\end{equation}
Since $q+2>s^\F(K)+1$ we have $i^\ast=0$, so $p$ is injective. By the commutativity of (\ref{eq:commdiag}) we get $j^\ast(\hat{a})\not=0$.

If there is $\hat{a}\in R(\alpha)$ with $j^\ast(\hat{a})\not=0$, it follows directly from (\ref{eq:commdiag}) and the definition that $s^\F(K)+1$ is $\alpha$-half-full. By the discussion above we have $r^\alpha_+(K)=s^\F(K)+2$.

Now assume that $s^\alpha_+(K)=s^\F(K)+2$, which means that there are $\hat{a},\hat{b}\in S(\alpha)$ with $\hat{a}=p(a)$, $\hat{b}=p(b)$ and $i^\ast(a),i^\ast(b)$ generate $\BN^0(K;\F)$. Look again at the long exact sequence (\ref{eq:longexact}), but this time with $q+2=s^\F(K)+1$. This time $i^\ast$ is non-zero, but not surjective. So some linear combination of $\hat{a},\hat{b}$ is not in the kernel of $j^\ast$.

If there is $\hat{b}\in S(\alpha)$ with $j^\ast(\hat{b})\not=0$, let $\check{a}\in H^0(\mathcal{F}_{s^\F(K)+1};\F)$ be such that $0\not=i^\ast(\check{a})\in \BN^0(K;\F)$. Choose $b\in H^0(\mathcal{F}_{s^\F(K)-1};\F)$ with $p(b)=\hat{b}$, and let $a=i^\ast(\check{a})\in H^0(\mathcal{F}_{s^\F(K)-1};\F)$. Then
\[
 i^\ast(a)\in \ker p\colon \BN^0(K;\F) \to H^0(\mathcal{F}_{-\infty}/\mathcal{F}_{s^\F(K)+1};\F),
\]
while $i^\ast(b)$ is not. We then get the configuration for $\alpha$-fullness of $q=s^\F(K)-1$ using $a,b$ and $0,\hat{b}$.
\end{proof}

\subsection{The operation $\Sq^1$}
The first cohomology operation of interest, considered in \cite{MR3189434}, is $\Sq^1$, the first Steenrod square. It is shown in \cite[Thm.3]{MR3189434} that there is a knot $K$ with
\[
s^{\Sq^1}_+(K) \not= s^{\F_2}(K).
\]
The knot in question is $14n19265$, and the argument is based on the fact that
\[
 s^{\F_2}(K) \not=s^\Q(K)
\]
by a computation of C.\ Seed and the Khovanov cohomology of this knot.

The first Steenrod square can be expressed in terms of the Bockstein sequence corresponding to the short exact sequence
\[
 0\longrightarrow \Z/2\Z \longrightarrow \Z/4\Z \longrightarrow \Z/2\Z \longrightarrow 0
\]
and thus can be defined directly on the Khovanov complex.

Step 1 to Step 5 work equally well on the commutative ring $\Z/4\Z$, but we need to be a little bit careful with reading off the Khovanov cohomology with $\Z/4\Z$-coefficients, as there may be several non-zero entries in the coboundary between generators of the same quantum degree. The ring $\Z/4\Z$ is not a domain, so we do not have a Smith Normal Form, but the ring is uncomplicated enough so that we can bring the cochain complex into a form that resembles a Smith Normal Form.

\begin{definition}
A cochain complex $C^\ast$ is called \em $\Z/4\Z$-elementary\em, if there is $n\in \Z$ and $i\geq 0$ such that
\[
 C^k = \left\{ \begin{array}{cc} 0 & k\notin \{n,\ldots,n+i\} \\ \Z/4\Z & k\in \{n,\ldots,n+i\} \end{array}\right.
\]
and the coboundary $d^k\colon C^k \to C^{k+1}$ is multiplication by $2$ for $n\leq k < n+i$. We call $i$ the \em length \em of the $\Z/4\Z$-elementary complex.

Let $C^\ast$ be a free cochain complex over $\Z/4\Z$ and there is a filtration
\[
 \cdots \subset \mathcal{F}_nC^\ast \subset \mathcal{F}_{n-1} C^\ast\subset \cdots \subset C^\ast
\]
by subcomplexes. Then $C^\ast$ is said to be in \em filtered $\Z/4\Z$-normal form\em, if each quotient complex $\mathcal{F}_{n-1}C^\ast/\mathcal{F}_nC^\ast$ is a direct sum of $\Z/4\Z$-elementary cochain complexes.
\end{definition}

\begin{lemma}
Let $C^\ast$ be a free $\Z/4\Z$-cochain complex such that there exists a finite basis $\mathcal{B}$, and subsets
\[
 \emptyset =\mathcal{B}_{n+1} \subset \mathcal{B}_n \subset \cdots \subset \mathcal{B}_{n-m} = \mathcal{B}
\]
such that each $\mathcal{B}_{n-j}$ generates a free subcomplex $\mathcal{F}_{n-j}C^\ast$ of $C^\ast$. Then $C^\ast$ is chain homotopy equivalent to a free filtered cochain complex $D^\ast$ which is in filtered $\Z/4\Z$-normal form, such that the chain homotopy equivalence $i\colon C^\ast \to D^\ast$ restricts to chain homotopy equivalences $i|\colon \mathcal{F}_{n-j}C^\ast \to \mathcal{F}_{n-j}D^\ast$.
\end{lemma}

\begin{proof}
The proof is done by induction over $m$. We first show that $\mathcal{F}_n$ can be brought into $\Z/4\Z$-normal form using Gaussian elimination and handle slides involving the basis elements of $\mathcal{B}_n$ only.

Using Gaussian elimination we can ensure that the coboundary matrix of $\mathcal{F}_n$ has only entries $0$ and $2$. We then use handle slides, compare, for example, \cite{MR3848400}, to turn $\mathcal{F}_n$ into a direct sum of $\Z/4\Z$-elementary cochain complexes. This argument is similar to the one used in the proof of \cite[Thm.6.2]{MR3848400}, where a chain complex is brought into Smith Normal Form.
One has to be a bit careful with elementary complexes of length bigger than $1$, in that one has to give priority to longer length. We will omit the details, especially since the complexes arising for us come from $\Z$-cochain complexes via the tensor product, and can never have length bigger than $1$.

Once we have the statement for $\mathcal{F}_{n-j+1}$, we use Gaussian elimination on basis elements in $\mathcal{B}_{n-j}-\mathcal{B}_{n-j+1}$ and then handle slides on the remaining ones. As $\mathcal{F}_{n-j+1}$ is a subcomplex, its coboundary is not affected by these operations. We can therefore finish the argument by induction.
\end{proof}

Notice that the $\Z/4\Z$-cochain complex $C^{\ast,\ast}_n$ we get after the first five steps can be used to read off the Khovanov cohomology with $\Z/2\Z$ coefficients by simply counting the generators according to their quantum and homological degrees. From this we can obtain $s^{\F_2}(K)$ using Steps 6-10.

We now apply handle slides, which can be done algorithmically, to get the complex into filtered $\Z/4\Z$-normal form, where we can read off $\Sq^1$ on Khovanov cohomology. It simply corresponds to the $\Z/4\Z$-elementary summands of length $1$ in the quotient complexes.

To get the intersection with the image of $p$ we need to check which linear combinations remain cocycles when considered in $\mathcal{F}_q$ for $q=s^{\F_2}(K)\pm 1$.

Finally, we need to see if any of these cocycles survive after applying the coboundary $d^{-1}\colon \F_2\otimes C^{-1,\ast}_n \to \F_2\otimes C^{0,\ast}_n$. If this is the case, we add $2$ to $s^{\F_2}(K)$ in the appropriate refinement, according to Proposition \ref{prop:readoffrefine}.

In order to work out $\Sq^1$ we should get the correct calculation for the Khovanov cohomology in homological degrees $-1$ and $1$. Using Step 4' is therefore not appropriate. However, we can easily modify this step to only disregard objects of homological degree bigger than $2$ or smaller than $-2-n+i$.

\subsection{The operation $\Sq^2$}
The second Steenrod square requires the stable homotopy type of \cite{MR3230817}, or at least a $1$-flow category as in \cite{2017arXiv171001857L}. At the moment there is no analogue of \cite{MR2320156} for the stable homotopy type, so we need a global approach to an appropriate $1$-flow category. 
Using Proposition \ref{prop:glueddiags} one can get $1$-flow categories $\cC_{\mathrm{Kh}}^{q}(K)$ for a knot $K$, whose cochain complexes are connected to form a Bar-Natan complex.

We only need to keep track of objects in homological gradings between $-3$ and $3$, and use the cancellation technique of \cite{2017arXiv171001857L} for objects of the same quantum degree. This way we can get much smaller $1$-flow categories whose cochain complexes can be used to read off the Khovanov cohomology. 
To determine $R(\Sq^2)$ and $S(\Sq^2)$ we can look at the smaller $1$-flow category whose objects have homological degree between $-2$ and $0$.
This $1$-flow category can be brought algorithmically into ``Chang form'' as in \cite{2017arXiv171001798L}. This Chang form makes it possible to read off $R(\Sq^2)$ or $S(\Sq^2)$ in the same way as $R(\Sq^1)$ or $S(\Sq^1)$, which then makes it possible to calculate $s^{\Sq^2}(K)$.

\section{Calculations and remarks on efficiency}
\label{sec:examples}

Consider the graph in Figure \ref{fig:k14n19265hom}. We can interpret the vertices as generators for a cochain complex $C^{\ast,\ast}$ with homological and quantum degrees as indicated, and the coboundary determined by the edges. The gray edges without labels represent either $+1$ or $-1$ in the coboundary.
The exact value is not important as we only need to know whether we can cancel generators between such an edge.
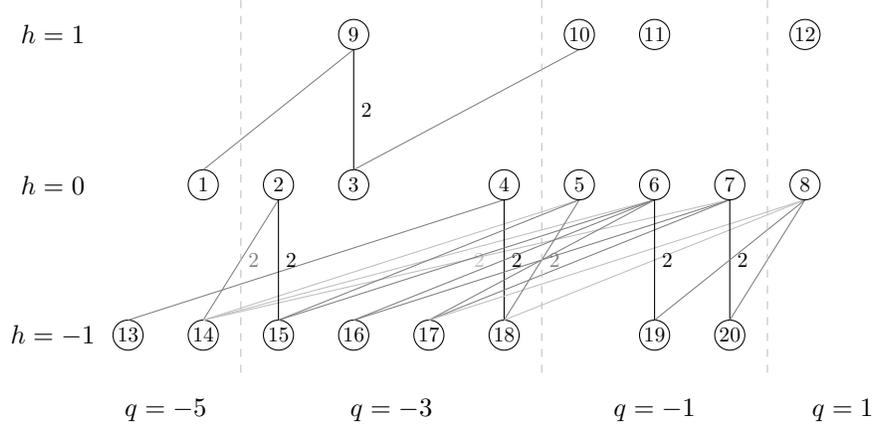
\begin{figure}[ht]
\begin{tikzpicture}
\draw (0,0) circle (0.2);
\node at (0,0) [scale = 0.8] {$13$};
\draw (1,0) circle (0.2);
\node at (1,0) [scale = 0.8] {$14$};
\draw (1,2) circle (0.2);
\node at (1,2) [scale = 0.8] {$1$};
\draw[color = gray] (1,2.2) -- (3,3.8);
\draw[color = gray] (0,0.2) -- (5,1.8);
\draw[color = gray] (1,0.2) -- node [right,scale =0.8] {$2$} (2,1.8);
\draw[color = lightgray] (1,0.2) -- (7,1.8);
\draw[color = lightgray] (1,0.2) -- (6,1.8);
\draw[color = lightgray] (1,0.2) -- node [right,scale =0.8] {$2$} (8,1.8);
\draw[dashed, color = lightgray] (1.5,-0.5) -- (1.5,4.5); 
\draw (2,0) circle (0.2);
\node at (2,0) [scale = 0.8] {$15$};
\draw (3,0) circle (0.2);
\node at (3,0) [scale = 0.8] {$16$};
\draw (4,0) circle (0.2);
\node at (4,0) [scale = 0.8] {$17$};
\draw (5,0) circle (0.2);
\node at (5,0) [scale = 0.8] {$18$};
\draw (2,2) circle (0.2);
\node at (2,2) [scale = 0.8] {$2$};
\draw (3,2) circle (0.2);
\node at (3,2) [scale = 0.8] {$3$};
\draw (5,2) circle (0.2);
\node at (5,2) [scale = 0.8] {$4$};
\draw (3,4) circle (0.2);
\node at (3,4) [scale = 0.8] {$9$};
\draw (2,0.2) -- node [right,scale =0.8] {$2$} (2,1.8);
\draw (3,2.2) -- node [right,scale =0.8] {$2$} (3,3.8);
\draw (5,0.2) -- node [right,scale =0.8] {$2$} (5,1.8);
\draw[color = gray] (3,2.2) -- (6,3.8);
\draw[color = gray] (2,0.2) -- (6,1.8);
\draw[color = gray] (2,0.2) -- (7,1.8);
\draw[color = gray] (3,0.2) -- (7,1.8);
\draw[color = gray] (3,0.2) -- (8,1.8);
\draw[color = gray] (4,0.2) -- (7,1.8);
\draw[color = gray] (4,0.2) -- (8,1.8);
\draw[color = lightgray] (4,0.2) -- (9,1.8);
\draw[color = gray] (5,0.2) -- node [right,scale =0.8] {$2$} (6,1.8);
\draw[color = lightgray] (5,0.2) -- (9,1.8);
\draw[dashed, color = lightgray] (5.5,-0.5) -- (5.5,4.5);
\draw (7,0) circle (0.2);
\node at (7,0) [scale = 0.8] {$19$};
\draw (8,0) circle (0.2);
\node at (8,0) [scale = 0.8] {$20$};
\draw (6,2) circle (0.2);
\node at (6,2) [scale = 0.8] {$5$};
\draw (7,2) circle (0.2);
\node at (7,2) [scale = 0.8] {$6$};
\draw (8,2) circle (0.2);
\node at (8,2) [scale = 0.8] {$7$};
\draw (6,4) circle (0.2);
\node at (6,4) [scale = 0.8] {$10$};
\draw (7,4) circle (0.2);
\node at (7,4) [scale = 0.8] {$11$};
\draw (8,0.2) -- node [right,scale =0.8] {$2$} (8,1.8);
\draw (7,0.2) -- node [right,scale =0.8] {$2$} (7,1.8);
\draw[color = gray] (7,0.2) -- (9,1.8);
\draw[color = gray] (8,0.2) -- (9,1.8);
\draw[dashed, color = lightgray] (8.5,-0.5) -- (8.5,4.5);
\draw (9,2) circle (0.2);
\node at (9,2) [scale = 0.8] {$8$};
\draw (9,4) circle (0.2);
\node at (9,4) [scale = 0.8] {$12$};
\node at (-1,0) {$h=-1$};
\node at (-1,2) {$h=0$};
\node at (-1,4) {$h=1$};
\node at (0.5,-1) {$q=-5$};
\node at (3.5,-1) {$q=-3$};
\node at (7,-1) {$q=-1$};
\node at (9.5,-1) {$q=1$};
\end{tikzpicture}
\caption{A cochain complex in filtered $\Z/4\Z$-normal form. The gray edges between $(2,14)$, $(5,18)$ and $(7,14)$ are labelled $2$.}
\label{fig:k14n19265hom}
\end{figure}

This is in fact obtained by using our computer programme on the knot $14n19265$ with $\Z/4\Z$ coefficients after Gaussian elimination and handle slides to turn it into filtered $\Z/4\Z$-normal form, but we can interpret this cochain complex over $\Z$ as well.

We can continue with our algorithm and perform steps 6 to 8. We first cancel generator \circlenumber{3} with \circlenumber{10}, as this has maximal quantum degree of those which can be cancelled with a generator of homological degree $1$. After that we cancel generator \circlenumber{1} with \circlenumber{9}.

We now have to cancel from below, starting with minimal quantum degree for which this is possible. So we can cancel \circlenumber{4} with \circlenumber{13}, while cancelling \circlenumber{2} with \circlenumber{14} is possible rationally, but not over $\F_2$. In fact, rationally we would have already cancelled all generators except for \circlenumber{1}, \circlenumber{5} and \circlenumber{8} in Step 5, so we can easily read off $s^\Q=0$.

Over $\F_2$, we can cancel \circlenumber{5} with \circlenumber{14}, and after that only one of \circlenumber{6} and \circlenumber{7}. Finally, we can cancel \circlenumber{8} and read off $s^{\F_2}=-2$.

To determine $r^{\Sq^1}_+$ observe that the image of $\Sq^1$ in quantum degree $-1$ is generated by \circlenumber{6} and \circlenumber{7}, and both are cocycles in $\F_2\otimes \mathcal{F}_{-1}$. But when viewed in the quotient complex $\F_2 \otimes \mathcal{F}_{-\infty} / \F_2\otimes \mathcal{F}_1$, we get \circlenumber{6} is cohomologous to \circlenumber{7}, and also to \circlenumber{5}, but this element survives. Hence $r^{\Sq^1}_+=0$ by Proposition \ref{prop:readoffrefine}.

Similarly, we see that $S(\Sq^1)$ is generated by \circlenumber{2} and \circlenumber{4}, but \circlenumber{4} is a coboundary in $\F_2 \otimes \mathcal{F}_{-\infty} / \F_2\otimes \mathcal{F}_{-1}$. Still, \circlenumber{2} survives, so $s^{\Sq^1}_+=0$.

If we turn Figure \ref{fig:k14n19265hom} upside down, we get $S(\Sq^1)=0$, and $R(\Sq^1)$ is generated by \circlenumber{3}. But \circlenumber{3} gets killed by \circlenumber{10} once we pass to the quotient complex. Overall, we get
\[
 s^{\Sq^1}(14n19265) = (0,0,-2,-2).
\]

\subsection{Calculations for $s^{\Sq^1}$}

The programme \verb+SKnotJob+ \cite{SKnotJob} was used to calculate the invariant $s^{\Sq^1}(K)$ for all prime knots with up to 16 crossings. We list all knots with at most 15 crossings for which this invariant has one entry different from $s^{\F_2}(K)$ in Table \ref{tab:tableLSs}. We also found 162 non-alternating prime knots with 16 crossings with this behaviour.

\begin{table}[ht]
\begin{tabular}{|l|c|c|c||l|c|c|c|}
\hline
Knot & $s^{\Sq^1}$ & $s^{\F_2}$ & $s^{\F_3}$ & Knot & $s^{\Sq^1}$ & $s^{\F_2}$ & $s^{\F_3}$ \\
\hline
\hline
$14n19265$ & $(0,0,-2,-2)$ & $ -2$ & $0$ &
$14n22180$ & $(2,2,0,0)$ & $2$ & $0$ \\
\hline
$15n040226$ & $(2,2,0,0)$ & $0$ & $2$ &
$15n041127$ & $(2,2,0,0)$ & $0$ & $2$ \\
\hline
$15n041140$ & $(0,0,-2,-2)$ & $0$ & $-2$ &
$15n048439$ & $(2,2,0,0)$ & $0$ & $2$ \\
\hline
$15n052310$ & $(2,2,0,0)$ & $0$ & $2$ &
$15n052477$ & $(2,2,0,0)$ & $0$ & $2$ \\
\hline
$15n052495$ & $(2,2,0,0)$ & $0$ & $2$ &
$15n053135$ & $(2,2,0,0)$ & $0$ & $2$ \\
\hline
$15n057674$ & $(0,0,-2,-2)$ & $0$ & $-2$ &
$15n059044$ & $(2,2,0,0)$ & $0$ & $2$ \\
\hline
$15n059184$ & $(0,0,-2,-2)$ & $0$ & $-2$ &
$15n083419$ & $(2,2,0,0)$ & $0$ & $2$ \\
\hline
$15n084460$ & $(4,4,2,2)$ & $2$ & $4$ &
$15n094892$ & $(2,2,0,0)$ & $0$ & $2$ \\
\hline
$15n115486$ & $(2,2,0,0)$ & $0$ & $2$ &
$15n116118$ & $(2,2,0,0)$ & $0$ & $2$ \\
\hline
$15n116363$ & $(2,2,0,0)$ & $0$ & $2$ &
$15n124915$ & $(2,2,0,0)$ & $0$ & $2$ \\
\hline
$15n127312$ & $(0,0,-2,-2)$ & $0$ & $-2$ &
$15n127580$ & $(0,0,-2,-2)$ & $0$ & $-2$ \\
\hline
$15n129563$ & $(0,0,-2,-2)$ & $0$ & $-2$ &
$15n130691$ & $(0,0,-2,-2)$ & $0$ & $-2$ \\
\hline
$15n132535$ & $(0,0,-2,-2)$ & $0$ & $-2$ &
$15n132615$ & $(0,0,-2,-2)$ & $0$ & $-2$ \\
\hline
$15n132623$ & $(2,2,0,0)$ & $0$ & $2$ &
$15n132672$ & $(2,2,0,0)$ & $0$ & $2$ \\
\hline
$15n132684$ & $(0,0,-2,-2)$ & $0$ & $-2$ &
$15n135086$ & $(2,2,0,0)$ & $0$ & $2$ \\
\hline
$15n135095$ & $(0,0,-2,-2)$ & $0$ & $-2$ &
$15n139104$ & $(2,2,0,0)$ & $0$ & $2$ \\
\hline
$15n140905$ & $(2,2,0,0)$ & $0$ & $2$ &
$15n141051$ & $(2,2,0,0)$ & $0$ & $2$ \\
\hline
$15n141061$ & $(0,0,-2,-2)$ & $0$ & $-2$ &
$15n141556$ & $(2,2,0,0)$ & $0$ & $2$ \\
\hline
$15n141560$ & $(0,0,-2,-2)$ & $0$ & $-2$ &
$15n149575$ & $(2,2,0,0)$ & $2$ & $0$ \\
\hline
\end{tabular}
\\[0.2cm]
\caption{Prime knots with non-standard $s^{\Sq^1}$ and at most $15$ crossings.}
\label{tab:tableLSs}
\end{table}

It is noticable that for all of these knots we also have $s^{\F_2}(K)\not= s^{\F_3}(K)$. In particular we do not get a better lower bound for the genus from $s^{\Sq^1}(K)$ than from the standard $s$-invariants. In Figure \ref{fig:hypothconf} we have a hypothetical configuration which would result in $s^{\F_2}=s^\Q=0$, but $s^{\Sq^1}=(0,2,0,-2)$. It would be interesting to know whether a knot could give rise to such a configuration. 

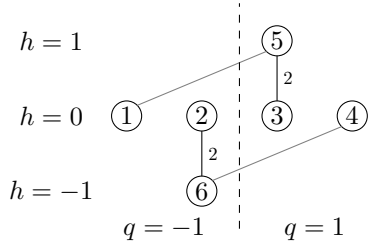
\begin{figure}[ht]
\begin{tikzpicture}
\node at (0,0) {$6$};
\draw (0,0) circle (0.2);
\node at (-1,1) {$1$};
\draw (-1,1) circle (0.2);
\node at (0,1) {$2$};
\draw (0,1) circle (0.2);
\node at (1,1) {$3$};
\draw (1,1) circle (0.2);
\node at (2,1) {$4$};
\draw (2,1) circle (0.2);
\node at (1,2) {$5$};
\draw (1,2) circle (0.2);
\draw (0,0.2) -- node [right,scale=0.7] {$2$} (0,0.8);
\draw (1,1.2) -- node [right,scale=0.7] {$2$} (1,1.8);
\draw [color=gray] (0.141,0.141) -- (1.859,0.859);
\draw [color=gray] (-0.859,1.141) -- (0.859,1.859);
\draw [dashed] (0.5,-0.5) -- (0.5,2.5);
\node at (-0.5,-0.5) {$q=-1$};
\node at (1.5,-0.5) {$q=1$};
\node at (-2,0) {$h=-1$};
\node at (-2,1) {$h=0$};
\node at (-2,2) {$h=1$};
\end{tikzpicture}
\caption{A cochain complex with unusual $s^{\Sq^1}$.}
\label{fig:hypothconf}
\end{figure}

It is known that $s^{\Sq^2}(K)$ gives better bounds for some knots \cite{MR3189434}, but we have not implemented an algorithm for this invariant yet.

\subsection{Efficiency}
Given an immersion of $S^1$ into the plane, which has at most double self-intersections, consider the knot diagrams obtained from this immersion by changing intersection points into over- or under-crossings. It is clear that our algorithm works better the further away the writhe of a diagram is from $0$. For example, if a diagram has $100$ crossings and half of them are positive, we will only start throwing away generators from the 52nd crossing onwards.
If there are only five negative (or positive) crossings, we start throwing away generators much sooner, and we will overall deal with much fewer generators.

In the extreme case of a positive knot $K$ it was already observed in \cite{MR2729272} that
\[
 s^\F(K) = −k + n + 1,
\]
where $n$ is the number of crossings and $k$ is the number of circles in the resolution sitting in homological degree $0$. To count the number of circles, one can scan through the crossings and check whether the $0$-resolution leads to a new circle. To some extend, this is what our algorithm does. Every time a new circle appears in the object in homological degree $0$, we can cancel it and never have more than one object in homological degree $0$.
One can check that we will not get more than $i$ objects of homological degree $1$ at the $i$-th crossing, and while this is not as efficient as simply counting circles, it is still done very fast.

In Table \ref{tab:calculations} we list times it took \verb+SKnotJob+ to perform various calculations.

\begin{table}[ht]
\begin{tabular}{|l|c|c|c|c|c|c|}
\hline
Knots & $s^{\F_2}$ & $s^{\F_3}$ & $s^{\Sq^1}$ & $\Kh(\cdot;\F_2)$ & $\Kh(\cdot;\Z)$ & Steps 1-5 \\
\hline
\hline
$3a1$--$13n5110$ & 0:57 & 1:09 & 1:22 & 1:17 & 1:31 & 1:35\\
\hline
$14n1$--$14n10000$ & 0:59 & 1:05 & 1:27 & 1:24 & 1:40 & 1:44 \\
\hline
$15n1$--$15n10000$ & 1:28 & 1:35 & 2:11 & 1:56 & 2:22 & 2:38 \\
\hline
$16n1$--$16n10000$ & 2:14 & 2:42 & 3:47 & 3:01 & 4:18 & 4:12 \\
\hline
\end{tabular}
\\[0.2cm]
\caption{Calculation times in minutes.}
\label{tab:calculations}
\end{table}

In the last column the algorithm runs through the first five steps with coefficients in $\F_2$ without disregarding any objects, which amounts to calculating the Khovanov cohomology with $\F_2$ coefficients. It is not surprising that this is slower than calculating Khovanov cohomology directly.

Alternating knots take noticably longer to calculate. By \cite[Thm.1.1]{MR2509750} the $s$-invariant cannot depend on the field in this case, and it furthermore agrees with the signature \cite[Thm.3]{MR2729272}. 

It would be possible to refine these comparisons by focussing on knots whose writhe is close to $0$ or away from it. Indeed, we did some experiments with 16 crossing knots and observed that for knots with writhe between $-2$ and $2$ the time improvements are less pronounced, but still there. 
Overall we observed that the direct calculation of the $s$-invariant is faster than going via Khovanov cohomology, but the scanning algorithm itself still occupies noticable time and memory.

The computations were done on a standard laptop. We do not claim the implementation is the most efficient, but the code used for the scanning algorithm on the Bar-Natan complex and the Khovanov complex is very similar. More efficient ways to compute Khovanov cohomology via the scanning algorithm should also result in more efficient ways to implement our algorithm.

\bibliography{KnotHomology}
\bibliographystyle{amsalpha}
\end{document}